\documentclass[a4paper]{article}
\usepackage{amsthm,amsmath,amssymb}
\usepackage{color}
\newtheorem{teor}{Theorem}[section]
\newtheorem{defin}[teor]{Definition}
\newtheorem{lemm}[teor]{Lemma}
\newtheorem{osse}[teor]{Remark}
\newtheorem{prop}[teor]{Proposition}
\newtheorem{defi}[teor]{Definition}
\newtheorem{coro}[teor]{Corollary}
\newtheorem{prob}[teor]{Problem}

\newcommand{\bele}{\begin{lemm}\begin{sl}}
\newcommand{\enle}{\end{sl}\end{lemm}}
\newcommand{\bedef}{\begin{defi}\begin{sl}}
\newcommand{\eddef}{\end{sl}\end{defi}}
\newcommand{\bete}{\begin{teor}\begin{sl}}
\newcommand{\ente}{\end{sl}\end{teor}}
\newcommand{\beos}{\begin{osse}\begin{rm}}
\newcommand{\eddos}{\end{rm}\end{osse}}
\newcommand{\bepr}{\begin{prop}\begin{sl}}
\newcommand{\empr}{\end{sl}\end{prop}}
\newcommand{\bepro}{\begin{prob}\begin{rm}}
\newcommand{\empro}{\end{rm}\end{prob}}
\newcommand{\bede}{\begin{defin}\begin{sl}}
\newcommand{\edde}{\end{sl}\end{defin}}
\newcommand{\beco}{\begin{coro}\begin{sl}}
\newcommand{\enco}{\end{sl}\end{coro}}

\newcommand{\quext}{\quad\text}

\newcommand{\de}{\partial}

\definecolor{grey}{rgb}{0.85,0.85,0.85}

\long\def\greybox#1{%
    \newbox\contentbox%
    \newbox\bkgdbox%
    \setbox\contentbox\hbox to \hsize{%
        \vtop{
            \kern\columnsep
            \hbox to \hsize{%
                \kern\columnsep%
                \advance\hsize by -2\columnsep%
                \setlength{\textwidth}{\hsize}%
                \vbox{
                    \parskip=\baselineskip
                    \parindent=0bp
                    #1
                }%
                \kern\columnsep%
            }%
            \kern\columnsep%
        }%
    }%
    \setbox\bkgdbox\vbox{
        \color{grey}
        \hrule width  \wd\contentbox %
               height \ht\contentbox %
               depth  \dp\contentbox
        \color{black}
    }%
    \wd\bkgdbox=0bp%
    \vbox{\hbox to \hsize{\box\bkgdbox\box\contentbox}}%
    \vskip\baselineskip%
}

\newcommand{\RR}{\mathbb{R}}

\newcommand{\beeq}[1]{\begin{equation}\label{#1}}
\newcommand{\eddeq}{\end{equation}}

\newcommand{\beeqa}[1]{\begin{eqnarray}\label{#1}}
\newcommand{\eddeqa}{\end{eqnarray}}

\newcommand{\beal}[1]{\begin{align}\label{#1}}
\newcommand{\eddal}{\end{align}}

\newcommand{\bespl}[1]{\begin{split}\label{#1}}
\newcommand{\edspl}{\end{split}}

\newcommand{\bega}[1]{\begin{gather}\label{#1}}
\newcommand{\edga}{\end{gather}}

\newcommand{\beeqax}{\begin{eqnarray*}}
\newcommand{\eddeqax}{\end{eqnarray*}}

\def\qed{\ifmmode 
  \else \leavevmode\unskip\penalty9999 \hbox{}\nobreak\hfill
  \fi
  \quad\hbox{\hskip.5em\vrule width.4em height.6em depth.05em\hskip.1em}}
\def\endproofsym{\qed}
\renewenvironment{proof}[1][Proof]{\trivlist\item[\hskip\labelsep{\hskip0pt
    {\normalfont\scshape#1.}\hskip .321429\parindent}]\ignorespaces}
{\endproofsym\endtrivlist}
\def\endnobox{\def\endproofsym{}\end{proof}\def\endproofsym{\qed}}

\newcommand{\no}{\nonumber}

\newcommand{\beeqao}{\begin{eqnarray}\no}
\newcommand{\bealo}{\begin{align}\no}
\newcommand{\besplo}{\begin{split}\no}
\newcommand{\begao}{\begin{gather}\no}

\newcommand{\eps}{\varepsilon}

\newcommand{\duav}[1]{\langle{#1}\rangle}

\newcommand{\duas}[1]{(\!({#1})\!)}
\newcommand{\duavw}[1]{\langle\!\langle{#1}\rangle\!\rangle}
\newcommand{\duavwt}[1]{\langle\!\langle{#1}\rangle\!\rangle_t}

\newcommand{\ov}[1]{\overline{#1}}

\newcommand{\spaziopfz}{\big\{\bv\in H^s(0,T_0;\bH) \cap L^2(0,T_0;H^{\frac12+s}(\Omega;\RR^3)):~\bv|_{\Gamma}\in L^4(0,T_0;L^4(\Gamma;\RR^3))\big\} }

\newcommand{\Ssz}{{\mathcal S}_s(T_0)}

\newcommand{\Om}{\Omega}

\newcommand{\itt}{\int_0^t}

\newcommand{\io}{\int_\Omega}

\newcommand{\igc}{\int_{\Gamma_C}}

\newcommand{\ito}{\itt\!\io}

\newcommand{\iTT}{\int_0^T}

\newcommand{\epsi}{\varepsilon}
\newcommand{\ee}{^{\varepsilon}}

\def\R{\mathbb R}

\newcommand{\bep}{\boldsymbol{\varepsilon}}

\newcommand{\bfhi}{\boldsymbol{\fhi}}

\newcommand{\bH}{\boldsymbol{H}}

\newcommand{\bg}{\boldsymbol{g}}
\newcommand{\bbf}{\boldsymbol{f}}

\newcommand{\bphi}{\boldsymbol{\varphi}}
\newcommand{\bpsi}{\boldsymbol{\psi}}
\newcommand{\bm}{\boldsymbol{m}}
\newcommand{\bt}{\boldsymbol{t}}
\newcommand{\bsig}{\boldsymbol{\sigma}}

\newcommand{\bzero}{\boldsymbol{0}}
\newcommand{\bv}{\boldsymbol{v}}
\newcommand{\vv}{\boldsymbol{v}}
\newcommand{\bV}{\boldsymbol{V}}

\newcommand{\fhi}{\varphi}

\newcommand{\lhs}{left hand side}
\newcommand{\rhs}{right hand side}

\DeclareMathOperator{\dive}{div}
\DeclareMathOperator{\deriv}{d}

\DeclareMathOperator{\dist}{d}

\DeclareMathOperator{\Id}{Id}

\DeclareMathOperator{\sign}{sign}

\let\TeXchi\chi
\def\chi{{\setbox0 \hbox{\mathsurround0pt
$\TeXchi$}\hbox{\raise\dp0 \copy0 }}}

\newcommand{\alfaciapo}{\widehat{\alpha}}

\newcommand{\gammaciapo}{\widehat{\gamma}}

\newcommand{\betaciapo}{\widehat{\beta}}

\newcommand{\calH}{{\mathcal H}}

\newcommand{\calT}{{\mathcal T}}

\newcommand{\calE}{{\mathcal E}}

\newcommand{\calV}{{\mathcal V}}

\newcommand{\baru}{\overline{\bu}}

\newcommand{\dit}{\deriv\!t}

\newcommand{\ddt}{\frac{\deriv\!{}}{\dit}}

\newcommand{\eet}{^\varepsilon_t}
\newcommand{\eett}{^\varepsilon_{tt}}

\newcommand{\bFormula}[1]{
\begin{equation} \label{#1}}
\newcommand{\eF}{\end{equation}}

\newcommand{\Ov}[1]{\overline{#1}}

\newcommand{\vu}{\vc{u}}
\newcommand{\vn}{\vc{n}}
\newcommand{\bn}{\vc{n}}
\newcommand{\vg}{\vc{g}}
\newcommand{\bu}{\vc{u}}
\newcommand{\vutt}{\vc{u}_{tt}}

\newcommand{\TE}{\mathbb{E}}
\newcommand{\TV}{\mathbb{V}}

\newcommand{\vc}[1]{{\bf #1}}
\newcommand{\calD}{{\mathcal D}}

\definecolor{violet}{rgb}{0.85,0.05,0.85}

\newenvironment{betti}{\color{red}}{\color{black}}
\newcommand{\bebe}{\begin{betti}}
\newcommand{\ebe}{\end{betti}}

\newenvironment{ele}{\color{blue}}{\color{black}}
\newcommand{\elena}{\begin{ele}}
\newcommand{\finele}{\end{ele}}

\newenvironment{riccardo}{\color{green}}{\color{black}}
\newcommand{\beri}{\begin{riccardo}}
\newcommand{\eri}{\end{riccardo}}

\newenvironment{giu}{\color{violet}}{\color{black}}
\newcommand{\giulio}{\begin{giu}}
\newcommand{\fingiu}{\end{giu}}

\numberwithin{equation}{section}
\begin{document}

\title{A contact problem for viscoelastic bodies with
 inertial effects and unilateral boundary constraints}

\author{%
Riccardo Scala\\
Dipartimento di Matematica, Universit\`a di Pavia,\\
Via Ferrata~1, 27100 Pavia, Italy\\
E-mail: {\tt riccardo.scala@unipv.it}
\and 
Giulio Schimperna\\
Dipartimento di Matematica, Universit\`a di Pavia,\\
Via Ferrata~1, 27100 Pavia, Italy\\
E-mail: {\tt giusch04@unipv.it}
}


\maketitle
\begin{abstract}
 We consider a viscoelastic body occupying a smooth bounded domain
 $\Omega\subset \RR^3$ under the effect of a volumic traction
 force $\bg$. The macroscopic displacement vector from the equilibrium
 configuration is denoted by $\vu$. Inertial effects are considered;
 hence the equation for $\vu$ contains the second order term
 $\vutt$. On a part $\Gamma_D$ of the boundary of $\Omega$, the body
 is anchored to a support and no displacement may
 occur; on a second part $\Gamma_N \subset \de \Omega$,
 the body can move freely; on a third portion $\Gamma_C \subset \de \Omega$,
 the body is in adhesive contact with a solid support.
 The boundary forces acting on $\Gamma_C$ 
 due to the action of elastic stresses are responsible
 for delamination, i.e., progressive failure of adhesive bonds.
 This phenomenon is mathematically represented by 
 a nonlinear ODE settled on $\Gamma_C$ and
 describing the evolution of the delamination order parameter~$z$.
 Following the lines of a new approach outlined in \cite{BRSS}
 and based on duality methods in Sobolev-Bochner spaces, 
 we define a suitable concept of weak 
 solution to the resulting PDE system.
 Correspondingly, we prove an existence result on finite
 time intervals of arbitrary length. 
\end{abstract}


\vspace{2mm}

\noindent {\bf AMS (MOS) subject clas\-si\-fi\-ca\-tion:} 
35L10, 74D10, 47H05, 46A20.


\section{Introduction}
\label{sec:intro}

Let $\Omega\subset\RR^3$ be a smooth bounded domain 
of boundary~$\Gamma$. We assume $\Omega$ to be occupied, during 
the fixed reference interval $(0,T)$, by a viscoelastic body.
No restriction is assumed on the final time $T>0$. 
The displacement vector with respect to the equilibrium 
configuration in $\Omega$ is noted by $\vu$.
Hence, we may assume that $\vu$ satisfies the following
equation
\begin{equation}\label{ela:intro}
  \vutt - \dive( \TV \bep(\vu_t) + \TE \bep(\vu) ) = \vg,
   \quext{in }\,\Omega,
\end{equation}
where, given a vector-valued function $\vv$,
$\bep(\vv)$ represents its symmetrized gradient;
$\vg$ is a volume force density; $\TV$ and $\TE$ are
the {\sl viscosity}\/ and {\sl elasticity}\/ tensors, respectively, 
assumed symmetric, nondegenerate, bounded,
and depending in a measurable way
on the variable $x\in \Omega$ (we refer to Section~\ref{sec:main}
below for the precise assumptions).
Inertial effects may occur and are mathematically
represented by the second order term~$\vutt$.

Let $\Gamma := \de \Omega$ be the boundary of $\Omega$ 
and let us assume $\Gamma$ to be decomposed as
$\Gamma = \ov{\Gamma_D} \cup \ov{\Gamma_N} \cup \ov{\Gamma_C}$,
where the portions $\Gamma_X$, for $X=D,N,C$, are assumed
to be relatively open in $\Gamma$, mutually disjoint, and to
have strictly positive $2$-dimensional measure. 
In addition to that, the distance between $\Gamma_C$
and $\Gamma_D$ is assumed to be strictly positive. In other words,
the contact surface $\Gamma_C$ and the support $\Gamma_D$ 
need to be strictly separated by $\Gamma_N$.
In the portion $\Gamma_D$ of the boundary, we suppose
the body to be anchored to a rigid support; hence no displacement
may occur, or, in other words, $\vu$ satisfies a homogeneous
Dirichlet condition. In the part $\Gamma_N$, we assume that the body 
is allowed to move freely, which corresponds to asking 
$\vu$ to satisfy a homogeneous Neumann boundary condition
(with standard adjustments we could equally consider
the nonhomogeneous case, corresponding to the physical
situation where a surface traction is exerted on~$\Gamma_N$).
Finally, on~$\Gamma_C$ (here, $C$ stands for ``contact''),
the body is assumed to be in adhesive contact with a hard surface,
like for instance a wall. This configuration has two 
effects: firstly, the (trace of the) displacement 
$\vu$ on $\Gamma_C$ must be directed towards the interior
of $\Omega$. Namely, we ask the following constraint to be satisfied:
\begin{equation}\label{const:intro}
  \vu\cdot \vn \le 0, \quext{on }\, \Gamma_C, 
\end{equation}
where $\vn$ is the outer normal unit vector to $\Omega$. 
The meaning of \eqref{const:intro}
is that the material in $\Omega$ may not penetrate the 
wall, but it may well detach from it. Secondly, the 
contact between the body in $\Omega$ and the wall
may cause ``damage'', i.e., loss of adhesive
properties. This phenomenon is often referred to as delamination process. 
Mathematically, it is described by means of 
a second (scalar) variable $z$, defined on~$\Gamma_C$
over the same time interval $(0,T)$.
The ``bonding function'' $z$ takes the form of an {\sl order parameter};
namely, $z(t,x)$ represents the fractional density of adhesive bonds 
that are active at the time 
$t\in(0,T)$ and at the point $x \in \Gamma_C$. In particular, 
when $z = 1$, there is total adhesion, whereas when $z = 0$ the 
bonds are completely broken. A value $z \in (0,1)$ denotes a partial
loss of adhesivity. In addition to that,
we assume that the bonds, once broken, cannot be repaired; namely,
we require the time derivative $z_t$ to be nonpositive at each 
$(t,x)\in(0,T) \times \Gamma_C$. 

In view of the above description, only the values $z\in[0,1]$
and $z_t\in(-\infty,0]$ are meaningful (or, as we will often say,
are ``physical''). Hence, the mathematical equation(s) for
$z$ must enforce in some way these constraints
(or, equivalently, exclude the ``unphysical'' configurations).
This scope may be reached by relying on the 
theory of {\sl maximal monotone
operators}\/ (cf.~the monographs \cite{At,Ba,Br}).
Indeed, the equation for $z$, settled on 
$\Gamma_C$, may take the form
\begin{equation}\label{dam:intro}
  \alpha(z_t) + z_t + \beta(z) \ni a - \frac12 | \vu |^2,
\end{equation}
where $\alpha$ and $\beta$ are {\sl monotone graphs}~in $\RR\times\RR$
entailing the required constraints. For instance, we may assume 
$\alpha=\de I_{(-\infty,0]}$,
i.e., the {\sl subdifferential}\/ of the {\sl indicator function}\/ 
of $(-\infty,0]$, whereas we may take $\beta=\de I_{[0,+\infty)}$,
i.e., the {\sl subdifferential}\/ of the {\sl indicator function}\/ 
of $[0,+\infty)$. Hence, $\alpha$ provides the nonpositivity
of $z_t$ (i.e., the {\sl irreversibility}\/ of damage),
whereas $\beta$ guarantees the nonnegativity
of $z$ (i.e., the fact that when the adhesive bonds are completely 
broken no further damage may occur). The constraint $z\le 1$ 
(i.e., the fact that the material cannot be more than completely 
integer) is automatically guaranteed by $z_t\le 0$ once it is $z\le 1$ 
at the initial time. 

The positive constant $a$ on the right hand side of~\eqref{dam:intro} 
(more generally, we may also admit $a$ to depend in a suitable 
way on $x\in \Gamma_C$) has the meaning of a threshold under which 
the elastic stresses are not strong enough to cause delamination.
Namely, when $\frac12 | \vu |^2$ is less than
$a$, the \rhs\ of \eqref{dam:intro} is nonnegative; hence no damage
is created. Instead, delamination may occur for $a-\frac12 | \vu |^2 < 0$.
Note also the term $z_t$ on the \lhs\ of \eqref{dam:intro}, meaning
we assume here the damage evolution to be {\sl rate-dependent}.

In order to get a closed system, one has to specify how the behavior
of the damage variable influences the behavior of $\vu$. This is made
explicit by the (up to now missing) boundary condition for $\vu$
on~$\Gamma_C$, which we assume in the form
\begin{equation}\label{bc:intro}
  - \big( \TV \bep(\vu_t) + \TE \bep(\vu) \big) \vn
   = \gamma( \vu \cdot \vn ) \vn + z \vu.
\end{equation}
Here, $\gamma$ is a third maximal monotone graph enforcing the 
non-penetration constraint (cf.~Remark~\ref{omog} below for 
further comments). In particular, if we take 
$\gamma = \alpha = \de I_{(-\infty,0]}$, we obtain that
(the trace of) $\vu$ is directed towards the interior
of $\Omega$, or, at most, we may have tangential displacements.
In the sequel we will actually allow for
more general assumptions on $\gamma$ (and this is why we decided 
to use a different symbol to denote it). 

Models for contact, delamination and damage in elastic media
are becoming very popular in the recent mathematical
literature. The evolution law of the adhesive damage variable~$z$ 
and its link with the displacement on $\Gamma_C$ is based on the so-called
concept of Fr\'emond delamination (see \cite{Fre}), here considered 
with the presence of nonnegligible viscosity of the adhesive (whose 
effects arise from the term $z_t$). Usually in the context of 
delamination models, the mechanical system consists of two (or more) 
elastic bodies attached upon an interface. Here we assume for simplicity 
to have only one body placed in $\Omega$ and attached to the support
in $\Gamma_C$; however we observe that our results could be 
extended to the general case by trivial generalizations. A dynamic model for
delamination without viscosity of the adhesive is also considered in~\cite{Rou}, 
where the evolution equation for the variable $\vu$ is a variant
of~\eqref{ela:intro} which takes higher order stresses into account.
A dynamic model where also thermal effects are considered
has been analyzed in \cite{RoRou}. Other related models, among many, can be found
in \cite{BBR1,BBR2,KMS,RaCaCo,RoRou2,RoTho}. A model coupling
\eqref{ela:intro} with an equation for $z$ similar to \eqref{dam:intro}
including viscosity can be found in \cite{Skadel}.  
The main difference between our model and the one in \cite{Skadel} 
stands however in the presence of the unilateral constraint \eqref{const:intro}, 
which represents also the main mathematical difficulty occurring here.
Actually, when one considers mechanical models with inertial effects
(i.e., containing the second order term $\vutt$), enforcing the constraint
\eqref{const:intro} by the methods of monotone operator theory
usually gives rise to regularity issues; namely, interpreting 
monotone graphs as monotone operators in the usual $L^2$ (or $L^p$)
framework is generally out of reach. For this reason, 
in most of the related literature, this difficulty is overcome by restating
the equation containing the constraint as a variational inequality.
This is the case, for instance, of the recent papers \cite{BBK,CSR}.
In \cite{BBK}, existence of a solution to a system closely related
to \eqref{ela:intro}+\eqref{bc:intro} (hence, the damage variable is 
not explicitly considered there) is proved by discretization techniques.
In \cite{CSR}, well posedness of a system also accounting for the
evolution of $z$ is proved by restating both equations as
variational inequalities. It is worth observing that,
differently from here, in \cite{CSR} the evolution of $z$ is assumed in 
to be {\sl reversible}, corresponding to the choice $\alpha\equiv0$ 
in the equivalent of~\eqref{dam:intro}; moreover, the quadratic 
term on the \rhs\ of the same relation is truncated (this simplifies
the proof of existence compared to our case).

In the present paper, we actually prefer to follow a partially different
approach, based on our recent work \cite{BRSS} 
(in collaboration with E.~Bonetti and E.~Rocca), where 
a strongly damped wave equation for a real valued variable $u$ 
containing a general constraint term is analyzed. 
The basic idea stands in writing a weak formulation where 
test-functions are chosen in a  ``parabolic'' 
Sobolev-Bochner space $\calV$ (in particular, in \cite{BRSS}, 
$\calV=H^1(0,T;L^2(\Omega)) \cap L^2(0,T;H^1(\Omega))$).
In that setting, the monotone graph
providing the constraint is restated as an operator acting
in the duality between $\calV$ and $\calV'$. Although 
it can be shown rather easily that this reformulation is 
factually equivalent to the statement as a (suitable) 
variational inequality, it presents a number of
notable advantages: firstly, it clarifies the regularity
of the constraint term, which usually corresponds
to a physical quantity; secondly, it permits us to prove 
further properties of solutions, like the energy inequality,
or to analyze the long-time behavior, which will be the
object of a forthcoming work. Finally, in this setting
we can still take advantage of the basic tools of 
monotone operator theory in Hilbert spaces
in order to prove existence.

It is worth pointing out a further mathematical difficulty
of the present model compared to similar ones. Namely, the coupling
term $|\vu|^2/2$ in \eqref{dam:intro} has a sort of critical growth
in space dimension~$3$. This is due to the {\sl mixed}\/
boundary conditions complementing~\eqref{ela:intro}, 
where the ``Neumann'' part \eqref{bc:intro} additionally accounts 
for the nonsmooth constraint term. In this setting,
the best space regularity we can hope to obtain for $\vu$ is 
$H^1$. This translates into an $L^4$-regularity of 
the trace on $\Gamma$ and into a (maximal) $L^2$-regularity of 
$|\vu|^2/2$. In view of this fact, the identification of 
the terms $\alpha(z_t)$ and $\beta(z)$ in the ``doubly nonlinear'' equation 
\eqref{dam:intro} is somehow nontrivial and requires a careful combination
of monotonicity and semicontinuity tools.

We finally point out that evolution equations with inertial terms and 
{\sl bilateral}\/ constraints, like for instance $\vu\cdot\bn=0$ on $\Gamma_C$, 
(corresponding to the case where only shear displacements are allowed)
are mathematically much simpler to deal with. These models arise, for instance, 
in physical systems subjected to high pressure. Such a behavior is 
often referred to as ``Mode~II'' evolution, in constrast to ``Mode~I''
evolutions, that are those where the constraint is unilateral, as
in~\eqref{const:intro}. 

The remainder of the paper is organized as follows: in the next section we 
introduce our assumptions on coefficients and data and state a mathematically
rigorous weak formulation together with our main result. This 
states existence of at least one weak solution of suitable regularity. 
The proof occupies the remainder of the paper. In particular, 
in Section~\ref{sec:appro} a regularized problem
is introduced and existence of a local solution to it is shown by means 
of a (Schauder) fixed point argument. 
Next, in Section~\ref{sec:limit}, the approximation
is removed by means of suitable a-priori estimates and compactness
methods. Other qualitative properties of weak solutions, like
the energy inequality, are also discussed there.


\section{Assumptions and main results}
\label{sec:main}

We set $H:=L^2(\Omega)$ and $\bH:=L^2(\Omega)^3$. 
%
%
Moreover, for $k\ge 1$, we introduce
\begin{align}\label{def:bHD}
  & \bH^k_D: = \big\{u\in H^k(\Om;\RR^3): u=0\text{ on } \Gamma_D\big\},
\end{align}
and we denote by $\bH^{-k}_D$ the respective dual spaces. 
We put $\bV:=\bH^1_D$, i.e.,
\begin{equation}\label{def:bV}
 \bV:=\big\{u\in H^1(\Om;\RR^3):u=0\text{ on }\Gamma_D \big\}.
\end{equation}
We also set $\bV_0:=H^1_0(\Omega;\RR^3)$ and recall that 
$\bV_0'=H^{-1}(\Omega;\RR^3)$. The spaces $\bV$ and $\bV_0$
are seen as (closed) subspaces of $H^1(\Om;\RR^3)$ 
(and in particular they inherit its norm).
The duality between $\bV$ and $\bV'$ will be indicated
by $\duav{\cdot,\cdot}$.

In the sequel we shall frequently use the continuity
of the trace operator
\begin{equation}\label{cont:trace}
   \text{from $H^1(\Omega)$ to $H^{1/2}(\Gamma_C)$}
   \quext{and from $H^1(\Omega)$ to $L^4(\Gamma_C)$}
\end{equation}
and its vector analogue. Moreover, the trace operator will be 
generally omitted in the notation; namely, functions defined in $\Omega$ 
and their traces on $\Gamma_C$ will be indicated by the same letters.

\smallskip

\noindent%
{\bf Strong formulation.}~~This can be stated as:
\begin{align}\label{ela}
  &  \vutt - \dive( \TV \bep(\vu_t) + \TE \bep(\vu) ) = \vg,
   \quext{in }\,(0,T)\times \Omega,\\
 \label{dam}
  & \alpha(z_t) + z_t + \beta(z) \ni a - \frac12 | \vu |^2,
  \quext{on }\,(0,T)\times \Gamma_C,\\
 \label{bc}
  & - \big( \TV \bep(\vu_t) + \TE \bep(\vu) \big) \bn
   \in \gamma ( \vu \cdot \bn ) \bn + z \vu,
   \quext{on }\,(0,T)\times \Gamma_C,
\end{align}
complemented with the additional boundary conditions
\begin{align}\label{dir}
  & \vu = \bzero,
   \quext{on }\,(0,T)\times \Gamma_D,\\
 \label{neum}
  & \big( \TV \bep(\vu_t) + \TE \bep(\vu) \big) \bn = \bzero,
   \quext{on }\,(0,T)\times \Gamma_N,
\end{align}
and with the Cauchy conditions
\begin{equation}\label{iniz}
  \vu|_{t=0}= \vu_0, \quad
  \vu_t|_{t=0}= \vu_1, \quad
  z|_{t=0}= z_0,
\end{equation}
where the first two relations are assumed a.e.~in~$\Omega$,
while the third one is stated a.e.~on~$\Gamma_C$.
\beos\label{omog}
It is worth noting that we considered the {\sl homogeneous}\/
condition \eqref{neum} just for the sake of simplicity. Indeed, one 
could deal with the case of a nonzero boundary traction 
$\bbf$ on the \rhs\ with standard modifications. Moreover, it may be also worth emphasizing
that, if $\bsig:=\TV \bep(\vu_t) + \TE \bep(\vu)$ denotes the 
stress tensor, then the boundary relation \eqref{bc} could 
be split into its normal and tangential parts
as follows:
\begin{align}\label{bcnor}
  & - \sigma_{\bn} := - ( \bsig \bn ) \cdot \bn
   \in \gamma ( \vu \cdot \bn ) + z \vu \cdot \bn,
   \quext{on }\,(0,T)\times \Gamma_C,\\
 \label{bctan}
  & - \bsig_{\bt} :=  - \bsig \bn + \big( ( \bsig \bn ) \cdot \bn ) \bn
    = z \vu_{\bt} := z \big(\vu - (\vu \cdot \bn) \bn \big),
   \quext{on }\,(0,T)\times \Gamma_C.
\end{align}
This corresponds exactly to the conditions considered, e.g., 
in~\cite[(1.23-24)]{BBR2} for $\nu=0$, i.e., when
no friction is assumed to occur on~$\Gamma_C$.
\eddos
\noindent%
As explained in the introduction, we are not able to provide
a solution to the strong (pointwise) formulation of the system.
Consequently, we need a weaker notion of solution. To introduce it,
we start with stating our assumptions on coefficients and data:
\begin{itemize}
\item[(a)]~%
The tensors $\TV,\TE \in L^{\infty}(\Omega; \RR^{81})$
satisfy, a.e.~in~$\Omega$, the standard symmetry properties
\begin{equation}\label{symm}
  \TV_{ijkl}=\TV_{jikl}=\TV_{ijlk}=\TV_{klij}, \quad
  \TE_{ijkl}=\TE_{jikl}=\TE_{ijlk}=\TE_{klij}.
\end{equation}
Moreover, $\TV,\TE$ are assumed to be (uniformly in~$\Omega$)
strongly positive definite; namely, there exists a constant~$\kappa>0$
such that
\begin{equation}\label{elas}
  \TE(x) \bep : \bep 
   \ge \kappa | \bep |^2, \qquad
  \TV(x) \bep : \bep 
   \ge \kappa | \bep |^2,
\end{equation}
for a.e.~$x\in \Omega$ and any symmetric matrix $\bep\in \RR^9$. Hence,
in view of Poincar\'e's and Korn's inequalities, for any 
$\bv\in \bV$ there holds
\begin{equation}\label{korn}
  \io \TE \bep(\bv) : \bep(\bv) 
   \ge \kappa \| \bv \|^2_{\bV}, \qquad
  \io \TV \bep(\bv) : \bep(\bv) 
   \ge \kappa \| \bv \|^2_{\bV},
\end{equation}
for a (possibly different) constant $\kappa>0$. 
\item[(b)]~%
We set $\alpha=\de I_{(-\infty,0]}$; moreover,
we assume $\beta$ and $\gamma$ be {\sl maximal monotone graphs}\/ 
in~$\RR\times \RR$ such that $\ov{D(\beta)}=[0,+\infty)$ and 
$\ov{D(\gamma)}=(-\infty,0]$. In particular,
an admissible choice is $\beta=\de I_{[0,+\infty)}$
and $\gamma=\de I_{(-\infty,0]}$. We recall that the domain $D(b)$ of a
graph $b\subset \RR\times \RR$ is the set $\{r\in\RR:~b(r)\not=\emptyset\}$.
We denote as $\alfaciapo$, $\betaciapo$, $\gammaciapo$ suitable convex and lower
semicontinuous functions from $\RR$ to $(-\infty,+\infty]$ such that
$\alpha=\de\alfaciapo$, $\beta=\de\betaciapo$, 
$\gamma=\de\gammaciapo$. Hence, we have in particular
$\alfaciapo=I_{(-\infty,0]}$.
We also suppose that
\begin{align}\label{limbasso}
  & \betaciapo(r), \gammaciapo(r) 
    \ge 0 \quext{resp.~for all }\, 
    r\in D(\betaciapo), D(\gammaciapo).
\end{align}
\item[(c)]~%
The initial data satisfy
\begin{equation}\label{hp:u0}
  \vu_0 \in \bV, \quad
   \gammaciapo(\vu_0\cdot\bn) \in L^1(\Gamma_C), \quad
   \vu_1 \in \bH,
\end{equation}
together with
\begin{equation}\label{hp:z0}
  z_0 \in L^\infty(\Gamma_C), \quad
   0 \le z_0 \le 1~~\text{a.e.~on~}\,\Gamma_C, \quad
   \beta^0(z_0)\in L^2(\Gamma_C).
\end{equation}
Here, for $r\in D(\beta)$, $\beta^0(r)$ is the element of minimum 
absolute value in the set $\beta(r)$ (cf.~\cite{Br}). Using the definition
of subdifferential one may easily prove
that \eqref{hp:z0} implies in particular
\begin{equation}\label{hp:z02}
   \betaciapo(z_0)\in L^1(\Gamma_C).
\end{equation}
\item[(d)]~%
The volumic force satisfies $\vg\in L^2(0,T;\bV')$.
\item[(e)]~%
We let $\Omega$ be a smooth and bounded domain of $\RR^3$.
We assume that $\Gamma=\de\Omega$ satisfies 
$\Gamma = \ov{\Gamma_D} \cup \ov{\Gamma_N} \cup \ov{\Gamma_C}$,
where $\Gamma_X$, for $X=D,N,C$, are relatively open in $\Gamma$, 
mutually disjoint, and have strictly positive $2$-dimensional Hausdorff measure. 
Moreover we assume the distance $\dist(\Gamma_C,\Gamma_D)>0$ and we 
require that $\Gamma_C$ is smooth as a subset of $\Gamma$
and has at most finitely many connected components. Specifically, 
we assume that the boundary of $\Gamma_C$ in $\Gamma$ is an
at most finite union of curves of class $C^1$.
\end{itemize}
\beos\label{sualfa}
 It would also be possible to consider more general choices for $\alpha$.
 For instance, we may ask $\alpha$ to be a maximal monotone graph
 satisfying $\ov{D(\alpha)}=(-\infty,0]$, the analogue of~\eqref{limbasso},
 plus some additional conditions regarding the behavior near~$0$.
 However, this would give rise to a number of technical complications
 in the proof, whence we decided to restrict ourselves to the basic choice
 $\alpha=\de I_{(-\infty,0]}$. On the other hand, our somehow
 general assumptions on $\beta$ and $\gamma$ do not require any 
 additional technical work. 
\eddos
\beos\label{suigrafi}
 Assumption~\eqref{limbasso} essentially states that 
 $\beta$ and $\gamma$ must have some coercivity at $\infty$. 
 For what concerns $\beta$ this is in fact
 just a technical assumption, in view of the fact that
 we will prove that $z\le 1$ almost everywhere. However,
 it may be useful in the approximation. Note that the analogue
 of~\eqref{limbasso} also holds for $\alfaciapo=I_{(-\infty,0]}$.
\eddos
\noindent%
{\bf Energy functional.}~~%
System \eqref{ela}-\eqref{bc} has a natural variational formulation. Namely, it 
can be seen as a generalized gradient flow problem for a suitable 
{\sl energy functional}. It is worth pointing out this structure from
the very beginning. To this aim, we will obtain the 
{\sl energy estimate}\/ directly from the system equations.
Of course, such a procedure has just a formal character 
at this level since we have not yet specified
which is our notion of solution and the related regularity. That said,
we first test \eqref{ela} by $\vu_t$. Using also \eqref{symm} with
the boundary conditions \eqref{dir} on~$\Gamma_D$ and \eqref{neum} 
on~$\Gamma_N$, it is easy to obtain
\begin{equation}\label{en11}
  \ddt \io \bigg( \frac12 | \vu_t |^2 
   + \frac12 \TE \bep(\vu) : \bep(\vu) \bigg)
   + \io \TV \bep(\vu_t) : \bep(\vu_t)
   = \igc \Big( \big( \TV \bep(\vu_t) + \TE \bep(\vu) \big) \bn \Big) \cdot \vu_t
     + \duav{\bg,\bu_t}.
\end{equation}
Then, we test \eqref{dam} by $z_t$ and integrate over $\Gamma_C$. A 
simple integration by parts in time gives
\begin{equation}\label{en12}
  \ddt \igc \bigg( \betaciapo(z) - a z 
   + \frac12 z | \vu |^2 \bigg)
   + \igc \big( \alpha(z_t) + z_t \big) z_t
   = \igc z ( \vu \cdot \vu_t ).
\end{equation}
Next, to cancel the terms on the \rhs s 
of \eqref{en11}-\eqref{en12}, one scalarly multiplies
\eqref{bc} by $-\vu_t$ and integrates, to obtain
\begin{equation}\label{en13}
  \igc \Big( \big( \TV \bep(\vu_t) + \TE \bep(\vu) \big) \bn \Big) \cdot \vu_t
   = - \ddt \igc \gammaciapo ( \vu \cdot \vn )
   - \igc z ( \vu \cdot \vu_t ).
\end{equation}
Hence, taking the sum of \eqref{en11}, \eqref{en12}, \eqref{en13}, we (formally)
obtain the {\sl energy identity}
\begin{equation}\label{energy}
   \ddt\calE(\vu,\vu_t,z) + \calD(\vu_t,z_t) = \duav{\bg,\bu_t},
\end{equation}
with the energy functional $\calE=\calE(\vu,\vu_t,z)$
given by
\begin{equation}\label{defiE}
  \calE := \io \bigg( \frac12 | \vu_t |^2 
   + \frac12 \TE \bep(\vu) : \bep(\vu) \bigg)
  + \igc \bigg( \betaciapo(z) - a z 
   + \frac12 z | \vu |^2 
   + \gammaciapo ( \vu \cdot \vn ) \bigg)
\end{equation}
and the dissipation integral(s) $\calD=\calD(\vu_t,z_t)$ 
defined as
\begin{equation}\label{defiD}
  \calD :=  \io \TV \bep(\vu_t) : \bep(\vu_t) 
   + \igc \big( \alpha(z_t) + z_t \big) z_t.
\end{equation}
Note that, in view of assumptions~(a)-(d), both $\calE$ and $\calD$
enjoy suitable coercivity properties. Observe also that
the \rhs\ of \eqref{energy} accounts for the contribution of
external volumic forces.

Although relation~\eqref{energy} corresponds to a basic 
physical property of the model, the mathematical procedure 
we used is formal under many aspects. The main problem is related
to the occurrence of the nonsmooth {\sl multivalued graphs}\/ 
$\alpha$, $\beta$, $\gamma$. Hence, we will see in Thm.~\ref{exist}
below that, for weak solutions, we will be only able to (rigorously) prove
a weak version of \eqref{energy} in the form of an 
{\sl inequality}\/ (cf.~\eqref{en:ineq} below).

In order to state our precise concept of weak solution we start
with introducing some more functional spaces:
\begin{align}\label{calV}
  &\mathcal V:=H^1(0,T;\bV),\\
 \label{calH}
  &\mathcal H:=H^1(0,T;H^{\frac{1}{2}}(\Gamma_C)),
\end{align}
and we let $\mathcal V'$ and $\mathcal H'$ be the respective dual spaces. 
Moreover, for all $t\in(0,T]$, we set
\begin{align}\label{calVt}
  & \mathcal V_t:=H^1(0,t;\bV),\\
 \label{calHt}
  & \mathcal H_t:=H^1(0,t;H^{\frac{1}{2}}(\Gamma_C)),
\end{align}
with the dual spaces $\mathcal V_t'$ and $\mathcal H_t'$, respectively.
Note that $\calH$ is exactly the space of traces (on~$\Gamma_C$)
of the elements of $\calV$ (and similarly for $\calH_t$ and $\calV_t$).
In the sequel, we shall note as $\duavw{\cdot,\cdot}$ the duality pairings
with respect to both space and time variables. For instance, that
symbol may note the duality between $\mathcal V$ and $\mathcal V'$ or also
that between $\mathcal H$ and $\mathcal H'$. When working on subintervals
$(0,t)$, $t\le T$, we will use the notation $\duavwt{\cdot,\cdot}$
(e.g., that may denote the duality between $\mathcal V_t$ and $\mathcal V_t'$).
Not to weight up formulas, we will use the symbol $(\cdot,\cdot)$ 
for the scalar product in both $\bH$ and $L^2(\Gamma_C)$. 
The norms in $\bH$ and $L^2(\Gamma_C)$ will be sometimes simply
noted by $\| \cdot \|$. The double brackets 
$(\!(\cdot,\cdot)\!)$ will represent the $L^2$-scalar product in
time-space variables (for instance, in $L^2(0,T;\bH)$ or in 
$L^2(0,T;L^2(\Gamma_C))$). On time subintervals, we will use the notation
$(\!(\cdot,\cdot)\!)_t$.

Next, we define the convex functional 
\begin{equation}\label{defiG}
  G: L^2(0,T;L^2(\Gamma_C)) \to [0,+\infty], \quad
   G(v):= \iTT\igc \gammaciapo(v).
\end{equation}
Then, if one considers the subdifferential
$\de G$ in the (Hilbert) space $L^2(0,T;L^2(\Gamma_C))$, it 
is well known that this coincides with
the realization of the graph $\gamma$. Namely, for 
$v,\eta \in L^2(0,T;L^2(\Gamma_C))$, one has
\begin{equation}\label{subdg}
  \eta \in \de G(v) ~ \Leftrightarrow ~
   \eta(t,x) \in \gamma(v(t,x))~~\text{a.e.~on }\,
     (0,T)\times \Gamma_C.
\end{equation}
Hence, in particular, $v$ complies with the constraint
represented by~$\gamma$. On the other hand, in our
specific situation, we will not be able to interpret 
$\gamma$ in the above sense, due to regularity lack coming
from the occurrence of the term $\vu_{tt}$.
For this reason, following the lines, e.g., of \cite{BRSS} (cf.~also
\cite{BCGG}), we provide a suitable relaxation 
of $\gamma$. To this aim, we identify $L^2(0,T;L^2(\Gamma_C))$ 
with its dual by means of the natural scalar product,
obtaining the chain of continuous and dense inclusions
\begin{equation}\label{triplet} 
  \calH \subset L^2(0,T;L^2(\Gamma_C)) 
   \sim L^2(0,T;L^2(\Gamma_C))' 
   \subset \calH'.
\end{equation}
Then, the above constructed {\sl Hilbert triplet}\/ permits us to 
relax $\gamma$ in the following sense: we define as $\gamma_w$ 
(where the subscript ``$w$'' stands for ``weak'')
the subdifferential of the restriction of $G$ 
to the space $\calH$ with respect to 
the duality between $\calH$ and $\calH'$. Namely, 
for $v\in \calH$ and $\eta \in \calH'$, we set
\begin{equation}\label{gammaw}
  \eta \in \gamma_w(v) ~ \Leftrightarrow ~
    \duavw{\eta, w - v} + G(v) \le G(w)
    ~~\text{for all }\,w \in \calH.
\end{equation}
A precise characterization of $\gamma_w$, which is a maximal
monotone operator from $\calH$ to $2^{\calH'}$, is carried out
in~\cite[Sec.~2]{BRSS} following the lines of results first proved 
in~\cite{brezisart} (see also \cite{BCGG}). 
Here we just mention the fact that, for any $v\in \calH$, there holds 
the inclusion $\gamma(v)\subset \gamma_w(v)$, 
which may however be strict.
On the other hand, once we know that $\eta\in \gamma_w(v)$,
then $v$ is still necessarily almost everywhere nonpositive
(hence, it satisfies the constraint); moreover, the inclusion
$\eta\in \gamma_w(v)$ has a precise ``measure-theoretic'' 
interpretation in terms of the original graph $\gamma$
(see \cite[Sec.~2]{BRSS} for more details).
The analogue of $\gamma_w$, noted with the same symbol for the
sake of simplicity, may be constructed also on the
space $\calH_t$, i.e., working on time subintervals.

\smallskip

We are now ready to introduce our concept of weak solution:
\begin{defin}\label{weaksol}
 Let $T>0$ and let\/ {\rm Assumptions (a)-(e)} hold.
 We say that $(\vu,\eta,z,\xi_1,\xi_2)$ is a weak solution to system 
 \eqref{ela}-\eqref{iniz} if 
 \begin{subequations}\label{reg}
 \begin{align}
  & \vu\in W^{1,\infty}(0,T;\bH)\cap H^1(0,T;\bV),\label{reg:u}\\
  & \vu_t\in H^1(0,T;\bV_0')\cap BV(0,T;\bH_D^{-2}),\label{reg:ut}\\
  & z\in W^{1,\infty}(0,T;L^2(\Gamma_C)),\label{reg:z}\\
  & \eta\in \mathcal H',\label{reg:eta}\\
  & \xi_1,\xi_2\in L^\infty(0,T;L^2(\Gamma_C))\label{reg:xi12}
 \end{align}
 \end{subequations}
 and the following properties are satisfied:
 \begin{itemize}
 \item[(i)] For all $\bphi\in \mathcal V$ it holds
 \begin{align}\label{eq1w}
  & (\vu_t(T),\bphi(T))-\duas{\vu_t,\bphi_t}+\duas{\TE \bep(\vu) , \bep(\bphi)}+\duas{\TV\bep(\vu_t), \bep(\bphi)}\nonumber\\
  & \mbox{} ~~~~~ 
   + \duavw{ \eta,\bphi \cdot \bn } + \iTT\igc z \vu \cdot\bphi
    = \duavw{ \bg, \bphi} + (\vu_1,\varphi(0)),
 \end{align}
 with the initial conditions \eqref{iniz}. Correspondingly, for every
 $t\in[0,T)$ there exists $\eta_{(t)}\in\mathcal H_t$ such that
 \begin{align}\label{eq1wt}
  & (\vu_t(t),\bphi(t)) - \duas{\vu_t,\varphi_t}_t+\duas{\TE \bep(\vu), \bep(\bphi)}_t
    + \duas{\TV\bep(\vu_t) , \bep(\bphi)}_t \nonumber\\
  & \mbox{} ~~~~~ +\duavw{ \eta_{(t)},\bphi \cdot \bn }_t + \itt\igc z \vu \cdot\bphi
    = \duavw{ \bg, \bphi}_t + (\vu_1,\varphi(0)),
 \end{align}
 for all $\bphi\in\mathcal V_t$. Moreover, the functionals $\eta$ and $\eta_{(t)}$
 are\/ {\rm compatible}, namely, if $\bphi\in\mathcal V_t$ satisfies 
 $\bphi(t)=0$ a.e.~on $\Om$, then we have
 \begin{equation}\label{compaeta}
   \duavw{ \eta_{(t)},\bphi \cdot \bn }_t=\duavw{ \eta,\Ov{\bphi} \cdot \bn },
 \end{equation}
 where $\Ov{\bphi}$ represents the trivial extension of $\bphi$ to $\mathcal V$ 
 (i.e., $\Ov{\bphi}(s,x)=0$ for $s\in[t,T]$ and $x\in\Omega$).
 \item[(ii)] For a.e.~$t\in(0,T)$ there holds
 \begin{equation}\label{eq:z}
     \xi_2(t) + z_t(t) + \xi_1(t) = a - \frac12 | \vu (t)|^2,
       \quext{a.e.~on }\,\Gamma_C.
 \end{equation}
 \item[(iii)] The following graphs inclusions hold true:
 \begin{align}
  & \xi_1\in\beta(z),\quext{a.e.~on }\,(0,T)\times \Gamma_C,\label{inc1}\\
  & \xi_2\in\alpha(z_t),\quext{a.e.~on }\,(0,T)\times \Gamma_C,\label{inc2}\\
  & \eta\in\gamma_w(\vu\cdot\vn).\label{incw}
 \end{align}
 Moreover, for all $t\in(0,T)$ we have  
 \begin{equation}\label{incwt}
   \eta_{(t)} \in\gamma_w ((\vu\cdot\vn)\llcorner_{(0,t)} ).
 \end{equation}
 \end{itemize}
\end{defin}
\noindent%
Note that, since $z$ and the term with $\eta$ in~\eqref{eq1w} 
are concentrated on $\Gamma_C$, if we choose $\bphi\in H^1(0,T;\bV_0)$
in \eqref{eq1w} and integrate by parts in time, we get back
\begin{align}
  & \langle u_{tt},\bphi\rangle
   + \int_\Om \TE \bep(\vu) : \bep(\bphi)
   + \int_\Om \TV\bep(\vu_t) : \bep(\bphi)
    = \langle g,\bphi\rangle, \label{eq1:epsw2}
\end{align}
a.e.~in~$t\in[0,T]$, where the first duality product makes sense
in view of the first~\eqref{reg:ut}. Hence, in particular 
\eqref{eq1w} implies \eqref{ela:intro} in the sense of distributions.

Let us now recall that the energy of the 
system was defined in~\eqref{defiE}. Moreover, analogously with 
\eqref{defiD}, we introduce the energy dissipation as
%
%
\begin{align}\label{new12}
 \mathcal D:= \int_\Omega \TV \bep(\vu_t) : \bep(\vu_t)
  + \int_{\Gamma_C} ( \xi_2 +  z_t ) z_t.
\end{align}
We are now ready to state our existence theorem, constituting
the main result of the present paper:
\begin{teor}\label{exist}
 Let $T>0$ and let {\rm Assumptions (a)-(e)} hold.
 Then there exists\/ {\rm at least} one weak solution $(\vu,\eta,z,\xi_1,\xi_2)$ to\/
 {\rm Problem~\eqref{ela}-\eqref{iniz}}, in the sense of\/ {\rm Def.~\ref{weaksol}}. 
 Moreover, for all times $t_2\in [0,T]$ and for a.e. $t_1\in[0,t_2)$,
 the following energy inequality holds:
 \begin{align}\label{en:ineq}
   & \mathcal E(t_2) + \int_{t_1}^{t_2}\mathcal D(\cdot)
    \leq \mathcal E(t_1)
    + \int_{t_1}^{t_2} \langle g,\vu_t\rangle.
 \end{align}
\end{teor}
\noindent%
The proof of the above result will occupy the remainder of the paper. 
%
%
%


\section{Regularized problem and local existence}
\label{sec:appro}

In this section we introduce a regularized version of our problem and prove existence
of a local (in time) solution by means of a suitable fixed point argument. In view of 
the fact that this procedure is standard under many aspects, we omit most 
details and just present the basic highlights.

First of all, for $\epsi\in(0,1)$ intended to go to $0$ in the limit,
we take suitable regularizations $\alpha\ee$, $\beta\ee$ and $\gamma\ee$ 
of the maximal monotone graphs $\alpha$, $\beta$ and $\gamma$. In particular,
we will consider the {\sl Yosida approximations}\/ defined, e.g., in \cite{Br},
to which we refer the reader for details. Here we just recall that
$\alpha\ee$, $\beta\ee$ and $\gamma\ee$ are 
monotone and Lipschitz continuous functions defined on the 
whole real line. Moreover,
%
%
%
%
they tend, respectively, to $\alpha$, $\beta$ and $\gamma$ in a suitable 
way, usually referred to as 
{\sl graph convergence}~in $\RR\times \RR$. Namely, for any
$[r;s]\in \alpha$ and any $\epsi\in(0,1)$, there exists 
$r\ee\in\RR$ such that $r\ee\to r$ and $\alpha\ee(r\ee)\to s$
in $\RR$ as $\epsi\searrow 0$, with analogous properties
holding for $\beta$ and~$\gamma$. 
In view of our choice $\alpha=\partial I_{(-\infty,0]}$,
we may explicitly compute
\begin{equation}\label{yosida}
  \alpha\ee(r) = \epsi^{-1} (r)^+, \quad
    A_\eps:=(\Id+\alpha^\eps)^{-1}(r)=
    - (r)^- + \frac{\eps}{\eps+1} (r)^+. 
\end{equation}
Notice that $A_\eps$ is a nonexpansive operator.

\smallskip

%

Our aim will be to solve, at least locally in time, a regularized statement,
which, in the strong form, can be written as follows:
\begin{align}\label{elaep}
  &  \vutt - \dive\big( \TV \bep(\vu_t) + \TE \bep(\vu) \big) = \vg,
   \quext{in }\,\Omega,\\
 \label{damep}
  & \alpha\ee(z_t) + z_t + \beta\ee(z) = a - \frac12 | \vu |^2,
  \quext{on }\,\Gamma_C,\\
 \label{bcep}
  & - \big( \TV \bep(\vu_t) + \TE \bep(\vu) \big) \bn
   = \gamma\ee ( \vu \cdot \bn ) \bn + (z)^+ \vu,
   \quext{on }\,\Gamma_C,
\end{align}
coupled with the initial conditions \eqref{iniz} and
the boundary conditions \eqref{dir}-\eqref{neum}.
The function $(z)^+$ in \eqref{bcep} denotes the positive part
of $z$. Actually, $z$ is not guaranteed to be nonnegative 
at the approximate level since the (smooth) function 
$\beta\ee$ does not enforce any constraint.

In fact, we will deal with a weak formulation of the above system,
to which we will apply Schauder's fixed point theorem. Hence, we 
start with introducing the fixed point space: for 
a small but otherwise arbitrary number $s\in (0,1/2)$,
and for $T_0\in(0,T]$ to be chosen at the end, we set
\begin{equation}\label{def:Ss}
  \Ssz:=\spaziopfz.
\end{equation}
The space $\Ssz$ is naturally endowed with the graph norm,
noted as $\| \cdot \|_{\Ssz}$ for brevity, which turns it into
a Banach space. Notice also that the trace of $\bv$ makes sense in
view of the assumed $H^{\frac12+s}$-space regularity.
Then, for some $M >0$, we consider the {\sl closed}\/ ball
\begin{equation}\label{def:BM}
  B_M:= \big\{ \bv \in \Ssz:~\| \bv \|_{\Ssz} \le M \big\}.
\end{equation}
Note that the choice of $M>0$ is essentially arbitrary. Its value
will in fact influence the resulting final time $T_0$, but 
this is irrelevant at the light of the subsequent uniform estimates.

\smallskip

The basic steps of our fixed point argument are carried out in the next
three lemmas.
\bele\label{lemma:1}
 Let $\vu_0$ satisfy~\eqref{hp:u0} and
 $z_0$ satisfy~\eqref{hp:z0}. More precisely, let us set
 \begin{align}\label{defiU}
   & U:= \| \vu_0 \|_{\bV} + \| \vu_1 \|_{\bH}
    + \| \gammaciapo(\vu_0\cdot\bn) \|_{L^1(\Gamma_C)},\\
  \label{defiZ}
   & Z:= \| z_0 \|_{L^2(\Gamma_C)}
     + \| \betaciapo(z_0) \|_{L^1(\Gamma_C)}.
 \end{align}
 Let also $\baru\in B_M$.
 Then there exists one and only one function $z$, with
 \begin{equation}\label{reg:z:fp}
    \| z \|_{L^\infty(0,T_0;L^2(\Gamma_C))}
    + \| z \|_{H^{1}(0,T_0;L^2(\Gamma_C))}
     \le Q(\epsi^{-1},Z,M),
 \end{equation}
 satisfying, a.e.~on~$(0,T_0)\times \Gamma_C$, the equation
 \begin{equation}\label{dam:pf}
    \alpha\ee(z_t) + z_t + \beta\ee(z) = a - \frac12 | \baru |^2,
 \end{equation}
 with the boundary condition $z|_{t=0}=z_0$.
 Here and below, $Q$ denotes a computable nonnegative-valued function,
 increasingly monotone in each of its arguments, whose expression
 may vary on occurrence.
\enle
\begin{proof}
Using \eqref{hp:z0}-\eqref{hp:z02} with the graph convergence
$\beta\ee\to \beta$, it is not difficult to prove that, at least for
$\epsi\in (0,1)$ small enough, there holds
\begin{equation}\label{z0Mee}
  \| z_0 \|_{L^2(\Gamma_C)} 
   + \| \betaciapo\ee(z_0) \|_{L^1(\Gamma_C)} 
   \le 2 Z.
\end{equation}
Then, existence of a solution $z$ can be proved by using standard 
existence results for ODE's. The regularity \eqref{reg:z:fp} can be 
inferred simply by testing \eqref{dam:pf} by $z_t$. 
Actually, integrating over $\Gamma_C$, we then obtain
\begin{align}\no
  \ddt \igc \betaciapo\ee(z)
   + \igc \alpha\ee(z_t) z_t
   + \| z_t \|_{L^2(\Gamma_C)}^2
  &  = \igc \Big( a - \frac12 | \baru |^2 \Big) z_t \\
 \label{l1:11}
   \le \frac 14 \| z_t \|_{L^2(\Gamma_C)}^2
    + c \big (1  + \| \baru \|_{L^4(\Gamma_C;\RR^3)}^4 \big).
\end{align}
Observe that the integration by parts is allowed in view
of the smoothness of $\beta\ee$.
%
%
Then, by the first~\eqref{yosida}, 
the second integral on the \lhs\ is nonegative. Hence, 
using also \eqref{z0Mee}, we infer
\begin{equation}\label{l1:13}
  \| z_t \|_{L^2(0,T_0;L^2(\Gamma_C))} 
   + \| \betaciapo\ee(z) \|_{L^\infty(0,T_0;L^1(\Gamma_C))} 
   \le Q(M,Z,\epsi^{-1}),
\end{equation}
whence, using again \eqref{z0Mee} to estimate $z$ from $z_t$,
we obtain \eqref{reg:z:fp}. Finally, to ensure uniqueness
of $z$ one can use standard contractive arguments.
For instance, if $z_1$ and $z_2$ are two solutions, one may test the difference of 
the corresponding equations \eqref{dam:pf} by the difference $(z_1-z_2)_t$
and use monotonicity and Lipschitz continuity of $\alpha\ee$ and 
$\beta\ee$ together with Gronwall's lemma. We omit details. 
\end{proof}
\bele\label{lemma:2}
 Let $\vu_0$, $z_0$, $M$, $\baru$ as above. Let also
 $z$ be the function provided by the previous lemma. Then there exists
 one and only one function $\vu$, with
 \begin{equation}\label{reg:u:fp}
    \| \bu \|_{W^{1,\infty}(0,T_0;\bH)}
    + \| \bu \|_{H^1(0,T_0;\bV)}
    + \| \bu \|_{L^\infty(0,T_0;\bV)}
     \le Q(\epsi^{-1},M,Z,U),
 \end{equation}
 satisfying, a.e.~on~$(0,T_0)$ and for any $\bphi \in \bV$, the  
 equation
 \begin{equation} \label{ela:pf}
    \duav{ \vutt,\bphi} + \io ( \TV \bep(\vu_t) + \TE \bep(\vu) ) : \bep(\bphi)
    + \igc \big( \gamma\ee ( \vu \cdot \bn ) \bn + (z)^+ \vu \big)\cdot \bphi 
    = \duav{ \bg, \bphi},
 \end{equation}
 together with the Cauchy conditions $\vu|_{t=0}=\vu_0$ and $\vu_t|_{t=0}=\vu_1$.
\enle
\begin{proof}
The weak formulation \eqref{ela:pf} is obtained from \eqref{elaep}
simply testing by  $\bphi \in \bV$, integrating by parts, and 
using the boundary condition \eqref{bcep}.
Then, existence of at least one solution can be proved by adapting the procedure
given, e.g., in \cite{BBR2}. Here we just reproduce the 
corresponding regularity estimate, which is needed in order to 
get \eqref{reg:u:fp}. Namely, we take $\bfhi=\vu_t$ in \eqref{ela:pf},
and, proceeding as in the Energy estimate detailed before, we get 
\begin{align}\no
  & \ddt \io \Big( \frac12 | \vu_t |^2 
   + \frac12 \TE \bep(\vu) : \bep(\vu) \Big)
   + \ddt \igc \gammaciapo\ee ( \vu \cdot \vn )
   + \io \TV \bep(\vu_t) : \bep(\vu_t) \\
 \no
  & \mbox{}~~~~~
   = - \igc (z)^+ ( \vu \cdot \vu_t )
   + \duav{ \bg, \bu_t}
   \le \| z \|_{L^2(\Gamma_C)} \| \vu \|_{L^4(\Gamma_C)} \| \vu_t \|_{L^4(\Gamma_C)} 
   + \| \bg \|_{\bV'} \| \bu_t \|_{\bV}\\
 \label{l2:11}  
  & \mbox{}~~~~~
   \le  Q(\epsi^{-1},M,Z) \| \vu \|_{\bV} \| \vu_t \|_{\bV} + \| \bg \|_{\bV'} \| \bu_t \|_{\bV}
   \le Q(\epsi^{-1},M,Z) \| \vu \|_{\bV}^2 + c \| \bg \|_{\bV'}^2
   + \frac{\kappa}2 \| \vu_t \|_{\bV}^2,
\end{align}
where $\kappa$ is the same constant as in \eqref{elas}.
Note that, to deduce the last inequalities, we used \eqref{reg:z:fp}
together with the trace theorem (cf.~\eqref{cont:trace})
and Young's inequality. Now, as before, from \eqref{defiU}
and the graph convergence $\gamma\ee\to \gamma$, we can easily
prove, at least for $\epsi\in(0,1)$ small enough,
\begin{equation}\label{defiU2}
  \| \vu_0 \|_{\bV} + \| \vu_1 \|_{\bH}
    + \| \gammaciapo\ee(\vu_0\cdot\bn) \|_{L^1(\Gamma_C)}
   \le 2 U.
\end{equation}
Hence, integrating \eqref{l2:11} in time, recalling 
\eqref{elas}-\eqref{korn}, and using 
Gronwall's lemma, we readily obtain that $\vu$ satisfies~\eqref{reg:u:fp}. 

To obtain uniqueness, we can proceed similarly as above. Namely, we 
assume to have two solutions $\vu_1$ and $\vu_2$ to \eqref{ela:pf},
take the difference of the corresponding equations, and substitute
$\bphi=(\vu_1-\vu_2)_t$ therein. Then, the Lipschitz continuity
of $\gamma\ee$ and the properties of the trace operator permit
us to get a contractive estimate via a procedure similar
to~\eqref{l2:11}.
\end{proof}
\bele\label{lemma:3}
 Let $\vu_0$, $z_0$, $M$, $\baru$ as above. Let also
 $z$ be the function provided by\/ {\rm Lemma~\ref{lemma:1}}
 and $\bu$ be the corresponding function provided by\/ 
 {\rm Lemma~\ref{lemma:2}}. Let us consider the map
 \begin{equation}\label{defiT}
    \calT: B_{M} \to W^{1,\infty}(0,T_0;\bH) \cap L^\infty(0,T_0;\bV), \quad
     \calT:\baru \mapsto \bu.
 \end{equation}
 %
 %
 %
 %
 Then, at least for $\epsi\in(0,1)$ sufficiently small, we can take
 $T_0\in (0,T]$, possibly depending on $\epsi$, $M$, $U$ and $Z$,
 such that the map $\calT$
 \begin{itemize}
 \item[(a)] takes values into $B_{M}$;
 \item[(b)] is continuous with respect to the (strong)
  topology of $\Ssz$;
 \item[(c)] maps $B_{M}$ into a compact subset of $\Ssz$.
 \end{itemize}
\enle
\begin{proof}
(a)~~Thanks to \eqref{reg:u:fp}, to \eqref{cont:trace}, and to 
standard embedding and trace theorems, we have
\begin{align}\label{u:fp2:1}
  \| \bu \|_{L^4(0,T_0;L^4(\Gamma;\RR^3))}
   & \le c T_0^{1/4} \| \bu \|_{L^\infty(0,T_0;L^4(\Gamma;\RR^3))}
   \le c T_0^{1/4} \| \bu \|_{L^\infty(0,T_0;\bV)}
   \le c T_0^{1/4} Q,\\
 \label{u:fp2:2}
  \| \bu \|_{H^s(0,T_0;\bH)}
   & \le c \| \bu \|_{H^1(0,T_0;\bH)}
   \le c T_0^{1/2} \| \bu \|_{W^{1,\infty}(0,T_0;\bH)}
   \le c T_0^{1/2} Q,\\
 \label{u:fp2:3}
  \| \bu \|_{L^2(0,T_0;H^{\frac12+s}(\Omega;\RR^3))}
   & \le c \| \bu \|_{L^2(0,T_0;\bV)}
   \le c T_0^{1/2} \| \bu \|_{L^\infty(0,T_0;\bV)}
   \le c T_0^{1/2} Q,
\end{align}
where we wrote $Q$ in place of $Q(\epsi^{-1},M,Z,U)$, for brevity,
and where $c>0$ are embedding constants {\sl independent of}\/ $T_0$.
Hence, we can choose $T_0$ sufficiently small, possibly depending
on $\epsi$, so that
\begin{equation}\label{l3:11}
  \| \bu \|_{\Ssz}
  \le c T_0^{1/4} Q(\epsi^{-1},M,Z,U) \le M,
\end{equation}
as desired. 
%

\smallskip

\noindent%
(b)~~Let $\{\baru_n\}\subset B_{M}$ and let 
$\baru_n\to \baru$ in $\Ssz$. Let also $z_n$ and $z$
be the corresponding functions given by Lemma~\ref{lemma:1}, let
$\vu_n=\calT(\baru_n)$ and let $\vu=\calT(\baru)$ 
be the corresponding solutions given by Lemma~\ref{lemma:2}.
We have to prove that 
\begin{equation}\label{l3:11b}
  \vu_n\to \vu \quext{strongly in }\, \Ssz.
\end{equation}
First, repeating, with the proper adaptations,
the uniqueness argument sketched in Lemma~\ref{lemma:1}
(and using in particular the Lipschitz continuity of $\alpha\ee$ and 
$\beta\ee$), we can easily show that 
\begin{equation}\label{l3:12}
  \lim_{n\nearrow \infty} \| z_n - z \|_{H^1(0,T_0;L^2(\Gamma_C))}
   = 0.
\end{equation}
Next, we work on equation \eqref{ela:pf}. 
Proceeding similarly with \eqref{l2:11} and performing standard
manipulations, it is not difficult to obtain
\begin{equation}\label{l3:13}
  \lim_{n\nearrow \infty} \big( \| \vu_n - \vu \|_{W^{1,\infty}(0,T_0;\bH)}
   + \| \vu_n - \vu \|_{L^\infty(0,T_0;\bV)} \big)
  = 0,
\end{equation}
This relation, also on account of~\eqref{cont:trace}, 
implies \eqref{l3:11b}, as desired.

\smallskip

\noindent%
(c)~~The proof is similar to the above one, but a bit more tricky. Indeed,
we still consider a sequence $\{\baru_n\}\subset B_{M}$, but we now 
just assume that 
\begin{equation}\label{weak:barun}
  \baru_n\to \baru \quext{weakly in }\, \Ssz.
\end{equation}
Then, with the same notation as above, we need to show that
at least a {\sl subsequence} of~$\{ \vu_n \}$ satisfies 
\eqref{l3:11b}. To prove this fact, we first observe that,
thanks to standard interpolation and (compact) embedding results,
there exists a (non-relabelled) subsequence of $n$ such that,
for some $p\in(1,2)$ depending on the choice of $s$,
\begin{equation}\label{l3:21}
  \baru_n\to \baru \quext{strongly in }\, L^{2p}(0,T_0;L^{2p}(\Gamma_C;\RR^3)),
\end{equation}
whence we have in particular
\begin{equation}\label{l3:22}
  |\baru_n|^2\to |\baru|^2 \quext{strongly in }\, L^p(0,T_0;L^p(\Gamma_C))
    \quext{and weakly in }\, L^2(0,T_0;L^2(\Gamma_C)).
\end{equation}
Let us now write \eqref{dam:pf} for the index $n$ and for the limit
(where $z$ is the solution corresponding to $\baru$),
and take the difference. Rearranging term and applying the 
inverse operator $(\Id + \alpha\ee)^{-1}$, we get the relation
\begin{equation}\label{l3:23}
  ( z_n - z )_t = (\Id + \alpha\ee)^{-1} 
     \Big( - \frac12 |\baru_n|^2 + \frac12 |\baru|^2
      - \beta\ee(z_n) + \beta\ee(z) \Big).     
\end{equation}
Then, testing by $|z_n - z|^{p-1}\sign(z_n-z)$ and
using the Lipschitz continuity of $(\Id + \alpha\ee)^{-1}$
and of $\beta\ee$ together with H\"older's and Young's 
inequalities, it is not difficult to arrive at
\begin{equation}\label{l3:24}
  \frac1p \ddt \| z_n - z \|_{L^p(\Gamma_C)}^p 
    \le C_\epsi \big( \| z_n - z \|_{L^p(\Gamma_C)}^p 
    + \| \baru_n - \baru \|_{L^{2p}(\Gamma_C)}^{p} 
     \| \baru_n + \baru \|_{L^{2p}(\Gamma_C)}^{p} \big).     
\end{equation}
Hence, applying Gronwall's lemma, we obtain 
\begin{equation}\label{l3:24b}
  z_n\to z \quext{strongly in }\, L^p(0,T_0;L^p(\Gamma_C)).
\end{equation}
Moreover, by the Lipschitz continuity of $(\Id + \alpha\ee)^{-1}$ 
and of $\beta\ee$, we can easily prove that the analogue of 
\eqref{l3:24b} holds also for $(z_n)_t$. 
%
%
Repeating the a priori estimates given at point~(b)  to obtain \eqref{l3:12}, 
we then conclude that
\begin{equation}\label{l3:25}
  z_n\to z \quext{weakly in }\, H^1(0,T_0;L^2(\Gamma_C)).
\end{equation}
Combining \eqref{l3:24b} and \eqref{l3:25} we also get
\begin{equation}\label{l3:25b}
  (z_n)^+\to (z)^+ \quext{strongly in }\, L^r(0,T_0;L^r(\Gamma_C))
   \quext{for every }\,r\in[1,2).
\end{equation}
Next, considering equation \eqref{ela:pf} with \rhs\
depending on $z_n$ and repeating the usual
energy estimate, it is easy to obtain
\begin{equation}\label{l3:31}
  \vu_n\to \vu \quext{weakly star in }\, W^{1,\infty}(0,T_0;\bH)
    \cap H^1(0,T;\bV)
\end{equation}
for some limit function $\vu$. In particular, due to standard
compact embedding results for vector-valued Sobolev spaces, this
entails
\begin{equation}\label{l3:31a}
  \vu_n\to \vu \quext{strongly in }\, 
   H^s(0,T_0;\bH) \cap L^2(0,T_0;H^{\frac12+s}(\Omega;\RR^3)).
\end{equation}
Hence, to get \eqref{l3:11b}, it remains to show that 
\begin{equation}\label{l3:31d}
  \vu_n\to \vu \quext{strongly in }\, 
   L^4(0,T_0;L^4(\Gamma;\RR^3)).
\end{equation}
This will be proved at the end. 
%
%
Preliminary, we may notice that $\vu$ solves 
the limit equation \eqref{ela:pf}. Actually, we have
\begin{equation}\label{l3:31b}
  (z_n)^+ \vu_n \to (z)^+ \vu \quext{weakly in }\, L^r(0,T_0;L^r(\Gamma_C;\RR^3))
    \quext{for some }\,r>1,
\end{equation}
as a consequence of \eqref{l3:25b}, \eqref{l3:31a}, and of the continuity 
of the trace operator from $H^{\frac12+s}(\Omega)$ 
to $H^{s}(\Gamma)$ for any $s\in(0,1/2)$.
Hence, it turns out that $\vu=\calT(\baru)$. To conclude the proof,
we then need to show \eqref{l3:31d}. More precisely,
we will reinforce \eqref{l3:31} proving that
\begin{equation}\label{l3:31str}
  \vu_n\to \vu \quext{{\bf strongly} in }\, W^{1,\infty}(0,T_0;\bH)
    \cap H^1(0,T;\bV).
\end{equation}
Then, \eqref{l3:31d} will follow from \eqref{cont:trace}. To get 
\eqref{l3:31str} we need to use a semicontinuity argument,
which we just sketch. Indeed, the same procedure will be repeated
in the next section under more restrictive assumptions.
We actually test \eqref{elaep}, written for $\vu_n$, by $\vu_n$,
and integrate. Then, using
\eqref{bcep} and performing some integration by parts, we arrive at
\begin{align}\nonumber
  & \ito \TE \bep(\vu^\eps) : \bep(\vu\ee)
   + \frac12 \io \TV \bep(\vu^\eps(t)) : \bep(\vu^\eps(t))
   = \frac12 \io \TV \bep(\vu_0) : \bep(\vu_0)
   - \ito \gamma^\eps(\vu^\eps\cdot\bn) \vu\ee \cdot \bn \\
 \label{semi01}  
  & \mbox{}~~~~~
   + \ito |\vu^\eps_{t}|^2
   - ( \vu\eet(t), \vu\ee(t) )
   + ( \vu_1, \vu_0 )
   - \itt\igc (z\ee)^+ |\vu\ee|^2
   + \itt \duav{ \bg, \vu\ee}.
\end{align}
We also test \eqref{elaep}, written for the limit $\vu$, by $\vu$,
obtaining an analogue relation. Then, we take the $\limsup$ of 
\eqref{semi01} at the level $n$ and we compare the outcome
with \eqref{semi01} written for $\vu$. Treating the terms on the
\rhs\ by owing to the smoothness of $\gamma\ee$ and by semicontinuity
tools (see the next section for details), we then obtain,
for every $t\in(0,T_0]$,
\begin{align}\nonumber
  & \limsup_{\epsi\searrow 0} \bigg( \ito \TE \bep(\vu^\eps) : \bep(\vu\ee)
   + \frac12 \io \TV \bep(\vu^\eps(t)) : \bep(\vu^\eps(t)) \bigg)\\
 \label{semi01b}  
  & \mbox{}~~~~~
    \le  \ito \TE \bep(\vu) : \bep(\vu)
   + \frac12 \io \TV \bep(\vu(t)) : \bep(\vu(t)).
\end{align}
Thanks to the symmetry and coercivity properties of the 
tensors $\TE$ and $\TV$ (cf.~assumption~(a)), we then get~\eqref{l3:31str},
whence \eqref{l3:31d}. Summarizing, we have
\begin{equation}\label{l3:31c}
  \calT(\baru_n) \to \calT(\baru)
   \quext{strongly in }\,\Ssz,
\end{equation}
which actually holds for the whole sequence $\baru_n$.
Hence, the map $\calT$ is compact. This concludes
the proof of the lemma.
\end{proof}
\noindent%
As a consequence of the three lemmas, we can apply Schauder's fixed point theorem
to the map~$\calT$, at least for $\epsi\in(0,1)$ sufficiently small. This
provides existence of {\sl at least}\/ one local solution $(\bu\ee,z\ee)$ to 
the approximate system. To be precise, $(\bu\ee,z\ee)$ satisfies,
a.e.~in~$(0,T_0)$, and for any $\bphi\in\bV$,
\begin{equation}\label{u:weak}
  \duav{ \vu\eett, \bphi}
   + \io \TE \bep(\vu\ee) : \bep(\bphi) 
   + \io \TV \bep(\vu\eet) : \bep(\bphi)
   + \igc \gamma\ee (\vu\ee \cdot \vn) \bphi\cdot \vn
   + \igc (z\ee)^+ \vu\ee \cdot \bphi
   = \duav{\bg, \bphi}.
\end{equation}
Moreover, \eqref{damep} holds a.e.~on $(0,T_0)$ and the couple~$(\bu\ee,z\ee)$ 
also complies with the initial conditions~\eqref{iniz}.
\beos\label{en:equal}
 It is worth remarking that $(\bu\ee,z\ee)$ also satisfies an
 approximate version of the energy {\sl equality}. Indeed, comparing
 terms in \eqref{u:weak}, one can easily prove that 
 $\vu\ee\in H^2(0,T_0;\bV')$. In particular, this is sufficient in order for
 $\bphi=\vu\eet\in L^2(0,T_0;\bV)$ to be an admissible test function
 in \eqref{u:weak}. Hence, testing also \eqref{damep} by $z\eet$
 and proceeding as in the last part of Section~\ref{sec:main},
 we may infer
 \begin{equation}\label{energyee}
   \ddt\calE\ee(\vu\ee,\vu\eet,z\ee) + \calD\ee(\vu\eet,z\eet) = \duav{\bg,\bu\eet}
    - \igc \frac12 (z\ee)^- \vu\ee\cdot\vu\eet,
    \quext{a.e.~in }\,(0,T_0).
 \end{equation}
 where $\calE\ee$ and $\calD\ee$ are the approximate energy and 
 dissipation functionals and the last term appears in view of the 
 occurrence of the positive part in~\eqref{bcep}.
 Note that here all integrations by parts are fully justified  
 in view of the fact that $\beta$ and $\gamma$ have been replaced
 by their regularized counterparts. Integrating \eqref{energyee}
 in time, we then also obtain the additional regularity
 \begin{equation}\label{C1}
   \vu\ee \in C^1([0,T_0];\bH).
 \end{equation}
\eddos


\section{Global existence for the original system}
\label{sec:limit}

In order to prove Theorem~\ref{exist}, we will show that, 
as $\epsi\rightarrow0$, the regularized solutions $(\vu\ee,z\ee)$
constructed before tend, in a suitable way 
and up to the extraction of a subsequence,
to a weak solution to the original problem. It is worth noting from
the very beginning that, at least in principle, the functions
$(\vu\ee,z\ee)$ are defined only on some subinterval $(0,T_0)$
possibly smaller than $(0,T)$ and also possibly depending on 
$\epsi$. However, in view of the fact that we shall derive 
a set of a-priori estimates that are independent of $T_0$, 
standard extension arguments imply that $(\vu\ee,z\ee)$ 
can in fact be extended to the whole of~$(0,T)$,
and the same will hold for the limit solution $(\vu,z)$.
Hence, in order to reduce technical complications, 
we shall directly assume with no loss of generality 
that $(\vu\ee,z\ee)$ are defined over $(0,T)$ already. 
In particular, the approximate energy equality
\eqref{energyee} and the related regularity \eqref{C1}
turn out to hold on the whole of~$(0,T)$.

That said, we proceed with the proof, which
is subdivided into various steps,
presented as separate lemmas. 
\begin{lemm}[Extension of boundary functions]\label{lemma0}
 Let $\bpsi\in H^{1/2}(\Gamma_C;\RR^3)$. Then, there exists 
 $\bphi=:R\bpsi\in \bV$ such that $\bphi|_{\Gamma_C} = \bpsi$ in the sense
 of traces. Moreover, the operator $R:\bpsi \mapsto \bphi$ is linear and 
 continuous from $H^{1/2}(\Gamma_C;\RR^3)$ to $\bV$, namely there
 exists $c_\Omega>0$, depending only on
 $\Omega$, $\Gamma_C$, $\Gamma_N$, $\Gamma_D$, such that
 \begin{equation}\label{rilev}
    \| \bphi \|_{\bV} 
     \le c \| \bpsi \|_{H^{1/2}(\Gamma_C;\RR^3)}. 
 \end{equation}
\end{lemm}
\begin{proof} 
We first observe that, in view of the regularity assumptions~(e)
on $\Omega$, using extension by reflection and cutoff
arguments, $\bpsi$ can 
be extended to a function $\widetilde\bpsi$ defined on 
the whole of $\Gamma$ in such a way that $\widetilde\bpsi$
is~$0$ a.e.~on~$\Gamma_D$ and
\begin{equation}\label{rilev1}
   \| \widetilde\bpsi \|_{H^{1/2}(\Gamma;\RR^3)}
    \le c \| \bpsi \|_{H^{1/2}(\Gamma_C;\RR^3)}. 
\end{equation}
Moreover, we can build the extension in such a way that
the map $\bpsi\mapsto \widetilde\bpsi$ is linear. 
As a second step, we construct $\bphi$ as the solution to the 
elliptic problem
\begin{equation}\label{rilev2}
  -\Delta \bphi = \bzero~~\text{in }\,\Omega, \qquad
   \bphi = \widetilde\bpsi~~\text{on }\,\Gamma.
\end{equation}
Then, the desired properties follow from \eqref{rilev1} 
and standard elliptic regularity results.
\end{proof}
\begin{lemm}[Step 1: first a priori estimate]\label{lemma1}
 There exists a constant $M>0$ independent of $\epsi\in(0,1)$ and 
 of $T_0\in(0,T]$ such that
 \begin{subequations}\label{est:uz}
 \begin{align}
   & \|\vu^\eps\|_{W^{1,\infty}(0,T;\bH)}\leq M,\label{est1}\\
   & \|\vu^\eps\|_{H^{1}(0,T;\bV)}\leq M,\label{est2}\\
   & \|\vu_{tt}^\eps\|_{L^2(0,T;\bV_0')}\leq M,\label{est4}\\
   & \|\vu_{t}^\eps\|_{BV(0,T;\bH^{-2}_D)}\leq M,\label{est4*}\\
   & \|z\ee\|_{H^{1}(0,T;L^2(\Gamma_C))}\leq M.\label{est3}
 \end{align}
 Moreover, we have
 \begin{align}
   &\|\gamma\ee(\vu^\eps\cdot \bn)\|_{L^1(0,T;L^1(\Gamma_C))}\leq M\label{est:beta:eps},\\
   &\|\gamma\ee(\vu^\eps\cdot \bn)\|_{\mathcal H'}\leq M, \label{beta:indual}
 \end{align}
 and, for all $t\in(0,T)$,
 \begin{align}
   & \|\gamma\ee(\vu^\eps\cdot \bn)\|_{\mathcal H_t'}\leq M. \label{beta:indualt}
 \end{align}
 \end{subequations}
\end{lemm}
\begin{proof}
Let us note as $\betaciapo\ee$ and $\gammaciapo\ee$ suitable antiderivatives
of $\beta\ee$ and $\gamma\ee$, respectively. Then,
using the standard relation $| b\ee (r) | \le |b^0(r)| $ holding 
for any maximal monotone graph $b\subset \RR\times \RR$ (cf.~\cite{Br}),
and recalling assumption~\eqref{limbasso}, it is not difficult to 
prove that, up to suitable choices of the integration constants,
we can assume $\betaciapo\ee$ and $\gammaciapo\ee$ to be nonnegative.

That said, let us take $\bphi=\vu\eet$ in \eqref{u:weak}, multiply \eqref{damep} 
by $\partial_t(z\ee)^+$ (i.e., the time derivative of the positive part
of $z\ee$), sum the resulting expressions, and integrate in time. Noting as 
$H$ the Heaviside function, we then observe that
\begin{equation}\label{Heav}
  \igc \beta\ee(z\ee) \partial_t(z\ee)^+
   = \igc \beta\ee(z\ee) H(z\ee) \partial_t(z\ee)
   = \ddt \igc B\ee(z\ee),
\end{equation}
where the function $B\ee(r)$ coincides with $\betaciapo\ee(r)$ for $r>0$
and is identically equal to $\betaciapo\ee(0)$ for $r\le 0$. 
Hence, we infer
\begin{align}\no
  & \frac{1}{2}\|\vu\eet(t)\|^2
   + \frac{1}{2} \int_\Om\TE \bep(\vu\ee(t)):\bep(\vu^\eps(t))
   + \int_0^t\int_\Om\TV \bep(\vu_t^\eps(t)):\bep(\vu_t^\eps(t))
   + \igc \gammaciapo\ee(\vu^\eps(t)\cdot \bn) \\
 \no  
  & \mbox{}~~~~~ 
   + \frac{1}{2}\igc (z^\eps)^+(t)|\vu^\eps(t)|^2
   + \int_0^t \| \de_t(z\ee)^+ \|^2
   + \int_0^t\igc \alpha^\eps(z_t^\eps) \de_t(z\ee)^+
   + \igc B\ee(z^\eps(t))\\
 \no
  & = a \itt \igc \partial_t(z\ee)^+
   + \int_0^t \langle \bg,\vu_t^\eps\rangle
   + \frac{1}{2} \|\vu_1\|^2
   + \frac{1}{2}\int_\Om \TE \bep(\vu_0):\bep(\vu_0)
   + \igc \gammaciapo^\eps(\vu_0\cdot \bn)
   + \igc B^\eps(z_0)\\
 \label{apr10} 
  & \mbox{}~~~~~
  + \frac{1}{2} \igc z_0 |\vu_0|^2
  \leq c 
   + \frac{1}{2}\int_0^t \| \de_t (z\ee)^+ \|^2
   + \frac{\kappa}{2} \int_0^t\|\nabla \vu_t^\eps\|^2
   + c \| \bg \|_{L^2(0,T;\bV')}^2,
\end{align}
where, in the last line, we have used Young's, Poincar\'e's and Korn's inequalities
together with the analogue of \eqref{z0Mee}, \eqref{defiU2}, and the fact that
\begin{equation}\label{eq:xi11}
 \frac{1}{2} \igc z_0 |\vu_0|^2
  \le \| z_0 \|_{L^2(\Gamma_C)} \| \vu_0 \|_{L^4(\Gamma_C;\RR^3)}^2 
  \le c \| z_0 \|_{L^2(\Gamma_C)} \| \vu_0 \|_{\bV}^2  
  \le c.
\end{equation}
Hence, recalling \eqref{elas}, we easily obtain \eqref{est1} and \eqref{est2}.

Let us now test \eqref{damep} by $z\eet$. Then, proceeding as 
before and noting that
\begin{align}\no
  \bigg| \itt\igc z\eet | \vu\ee |^2 \bigg|
   & \le \| z\eet \|_{L^2(0,T;L^2(\Gamma_C))} 
    \| \vu\ee \|_{L^4(0,T;L^4(\Gamma_C;\RR^3))}^2\\
 \label{apr10b}
    & \le \frac12 \| z\eet \|_{L^2(0,T;L^2(\Gamma_C))}^2
    + c \| \vu\ee \|_{L^4(0,T;\bV)}^4,
\end{align}
using \eqref{est2} to estimate the last term, we readily infer \eqref{est3}. 

Next, we choose $\bphi\in \bV_0$ in \eqref{u:weak}. 
Then, the boundary integrals go away and a simple comparison
of terms permits us to obtain \eqref{est4}.

Let us now prove \eqref{beta:indual}. To this aim, let
$\psi\in \mathcal H$. Then, in view of the regularity 
assumptions~(e), we can think the outer unit normal $\bn$ to~$\Omega$
on~$\Gamma_C$ to be the trace of a smooth function, denoted with the 
same symbol, defined on an open neighbourhood $\Lambda\subset \RR^3$
of $\Gamma_C$. As a consequence, $\bpsi:=\psi\bn$ lies 
in $H^1(0,T;H^{1/2}(\Gamma_C;\RR^3))$.
Hence, applying Lemma~\ref{lemma0}, we may
construct $\bphi=R\bpsi\in H^1(0,T;\bV)$ such that 
$\bphi|_{\Gamma_C}=\bpsi$) and 
\begin{equation}\label{rilev3}
  \| \bphi \|_{H^1(0,T;\bV)} 
   \le c \| \bpsi \|_{H^1(0,T;H^{1/2}(\Gamma_C;\RR^3))} 
   \le c \| \psi \|_{\calH}. 
\end{equation}
Taking such a $\bphi$ in \eqref{u:weak}, integrating in time,
performing suitable integrations by parts, comparing terms, 
and applying \eqref{elas} and Korn's inequality, we then have
\begin{align}\nonumber
  & |\duavw{ \gamma\ee(\vu^\eps\cdot \bn),\psi }|
   \leq \| \vu^\eps_t\|_{L^2(0,T;\bH)}\|\bphi_t\|_{L^2(0,T;\bH)}
   + \| \vu^\eps_t(T)\|_{\bH} \|\bphi(T)\|_{\bH}
   + \| \vu_1\|_{\bH} \| \bphi(0)\|_{\bH} \\
 \nonumber
   & \mbox{}~~~~~
    + c\| \nabla \vu^\eps_t\|_{L^2(0,T;\bH)} \|\nabla \bphi\|_{L^2(0,T;\bH)}
    + c\| \nabla \vu^\eps\|_{L^2(0,T;\bH)} \| \nabla \bphi\|_{L^2(0,T;\bH)}\\
 \nonumber
    & \mbox{}~~~~~
   + \| (z^\eps)^+ \vu^\eps\|_{L^2(0,T;L^2(\Gamma_C;\RR^3))} \| \bphi \|_{L^2(0,T;L^2(\Gamma_C;\RR^3))}
   +\| \bg \|_{L^2(0,T;\bV')} \|\bphi\|_{L^2(0,T;\bV)}\\
 \label{gamma:est}
   & \leq c \|\bphi\|_{\mathcal V} \leq c\|\psi\|_{\mathcal H},
\end{align}
where the constants are provided by \eqref{est1}-\eqref{est4},
\eqref{est3}, and the continuity of the trace operator. Hence, 
\eqref{beta:indual} follows. Repeating the same argument on subintervals,
we also obtain \eqref{beta:indualt}.

Finally, to get \eqref{est:beta:eps} let us apply Lemma~\ref{lemma0}
to the function $\bpsi=\bn$. As noted above, we may assume 
$\bn \in H^1(0,T;H^{1/2}(\Gamma_C;\RR^3))$. Hence, we obtain an extension
$\bm\in H^1(0,T;\bV)$ of $\bn$ to the whole domain $\Omega$. Let us
now plug $\bphi=\bu\ee + \bm$ into \eqref{u:weak}. Then, estimating
the resulting \rhs\ as in~\eqref{gamma:est}, it is not difficult
to arrive~at
\begin{equation}\label{dalbasso1}
  \itt\igc \gamma\ee(\vu^\eps\cdot \bn) (\bu\ee + \bm) \cdot \bn
    \leq c \| \bu\ee + \bm \|_{\mathcal V}
    \leq c,
\end{equation}
the last inequality following from \eqref{est2} and Lemma~\ref{lemma0}
applied to $\bn$. Then, we claim that the \lhs\ 
can be estimated from below as follows:
\begin{equation}\label{dalbasso2}
  \itt\igc \gamma\ee(\vu^\eps\cdot \bn) (\bu\ee + \bm) \cdot \bn
   = \itt\igc \gamma\ee(\vu^\eps\cdot \bn) (\bu\ee \cdot \bn + 1) 
   \ge \kappa_0 \| \gamma\ee(\vu^\eps\cdot \bn) \|_{L^1(0,T:L^1(\Gamma_C))} - c.
\end{equation}
with $\kappa_0>0$ and $c\ge 0$ independent of $\epsi$. 
To prove the last inequality in \eqref{dalbasso2} we can use 
the assumption $\ov{D(\gamma)}=(-\infty,0]$ with the coercivity
\eqref{limbasso}. Indeed, on the one hand this implies that,
for $r \ge -1/2$, there holds
\begin{equation}\label{dalbasso4}
  \gamma\ee(r) (r + 1) 
   \ge \frac12 | \gamma\ee (r) |.
\end{equation}
On the other hand, due to \eqref{limbasso}, 
either $\lim_{r\to -\infty} \gamma(r) < 0$
(and the same holds for $\gamma\ee$), whence
we can reason as in \eqref{dalbasso4} also for $r<<0$,
or it is $\lim_{r\to -\infty} \gamma(r) = 0$ 
(and, again, the same holds for $\gamma\ee$), so that there is
nothing to prove because in that case the $L^1$-norm 
of $\gamma\ee(\vu^\eps\cdot \bn)$ may only explode on 
the set where $\vu^\eps\cdot \bn \ge -1/2$. Hence,
we have proved \eqref{dalbasso2}, which, combined
with \eqref{dalbasso1}, gives \eqref{est:beta:eps}.

Finally, we take $\bphi\in \bH^2_D$
in \eqref{u:weak}. Then, noting that 
$\| \bphi \|_{L^\infty(\Gamma_C)} \le c \| \bphi \|_{\bH^2_D}$,
a comparison of terms in \eqref{u:weak}, together with estimate 
\eqref{est:beta:eps}, permits us to obtain \eqref{est4*},
which concludes the proof of the lemma.
\end{proof}
\begin{lemm}[Step 2: second a priori estimate]\label{lemma2}
 There exists a constant $M>0$ independent of $\epsi\in(0,1)$ and 
 of $T_0\in(0,T]$ such that
 \begin{subequations}\label{est:xi}
 \begin{align}
   &\|\beta^\eps(z^\eps)\|_{L^\infty(0,T;L^2(\Gamma_C))}\leq M,  \label{est:beta:z}\\
   &\|\alpha^\eps(z_t^\eps)\|_{L^\infty(0,T;L^2(\Gamma_C))}\leq M\label{est:alpha:zt},\\
   & \| z_t^\eps \|_{L^\infty(0,T;L^2(\Gamma_C))}\leq M\label{est:zt2}.
 \end{align}
 \end{subequations}
\end{lemm}
\begin{proof}
Let us multiply equation \eqref{dam:pf} by $\ddt \beta^\eps(z^\eps)=(\beta^\eps)'(z^\eps)z_t^\eps$, 
so to obtain
\begin{align}
 & \frac{1}{2} \ddt \igc |\beta^\eps(z^\eps)|^2
  + \igc (\beta^\eps)'(z^\eps) |z_t^\eps|^2
  + \igc (\beta^\eps)'(z^\eps)\alpha^\eps(z_t^\eps)z_t^\eps\nonumber\\
 \label{bg11}  
 & = \ddt \igc \Big(a-\frac{1}{2}|\vu^\eps|^2\Big)\beta^\eps(z^\eps)
  + \igc ( \vu^\eps\cdot \vu_t^\eps ) \beta^\eps(z^\eps),
\end{align}
whence, integrating over $(0,t)$, $0<t\le T$, and using that
$|\beta\ee(\cdot)|\le |\beta^0(\cdot)|$ together 
with assumption \eqref{hp:z0}, we get
\begin{align}
  & \frac{1}{4} \|\beta^\eps(z^\eps(t))\|_{L^2(\Gamma_C)}^2
   + \int_0^t \igc (\beta^\eps)'(z^\eps)|z_t^\eps|^2
   + \int_0^t \igc (\beta^\eps)'(z^\eps)\alpha^\eps(z_t^\eps)z_t^\eps \nonumber\\
 \nonumber
  & \mbox{}~~~~~\leq \frac{1}{2} \|\beta^\eps(z_0) \|_{L^2(\Gamma_C)}^2
   + \int_\Gamma \Big( a-\frac{1}{2}|\vu^\eps(t)|^2\Big)^2\\
 \no
  & \mbox{}~~~~~~~~~~ + \|\vu^\eps\|_{L^\infty(0,T;L^4(\Gamma_C;\RR^3))}
      \Big( \frac12\| \vu_t^\eps\|_{L^2(0,T;L^4(\Gamma_C;\RR^3))}^2 
          + \frac12 \|\beta^\eps(z^\eps)\|_{L^2(0,T;L^2(\Gamma_C))}^2\Big)\\
 \label{bg12}     
  & \mbox{}~~~~~ 
   \leq c \big( 1 + \|\beta^\eps(z^\eps)\|^2_{L^2(0,T;L^2(\Gamma_C))} \big),
\end{align}
where we have also used Young's inequality and the estimates \eqref{est:uz}. 
Now, since $(\beta^\eps)'\geq 0$ and $\alpha^\eps$ is monotone and satisfies
$\alpha^\eps(0)=0$, we may use Gronwall's lemma to obtain \eqref{est:beta:z}.
Moreover, applying the nonexpansive operator $(\Id + \alpha\ee)^{-1}$,
we may rewrite \eqref{dam:pf} in the form
\begin{equation}\label{l3:23b}
  z\ee_t = (\Id + \alpha\ee)^{-1} 
     \Big( a - \frac12 | \vu\ee |^2 - \beta\ee(z\ee) \Big).     
\end{equation}
Comparing terms in \eqref{l3:23b}, we then get \eqref{est:zt2}. 
With this information at disposal, we go back to~\eqref{dam:pf}
and a further comparison argument gives also \eqref{est:alpha:zt}.
\end{proof}
\begin{lemm}[Step 3: converging subsequence]\label{lemma3}
 There exist limit functions $(\vu,z,\eta,\xi_1,\xi_2)$ such that, 
 for a (non relabelled) subsequence of $\eps\rightarrow0$, there holds
 \begin{subequations}
 \begin{align}
   & \vu^\eps\rightharpoonup \vu \quext{weakly in }\,H^1(0,T;\bV)~~\text{and weakly star in }\,W^{1,\infty}(0,T;\bH),\label{conv1}\\
   & \vu_t^\eps\rightharpoonup \vu_t \quext{weakly in }\,H^1(0,T;\bV_0')~~\text{and weakly star in }\,BV(0,T;\bH_D^{-2}),\label{conv3}\\
   & z^\eps\rightarrow z \quext{weakly star in }\,W^{1,\infty}(0,T;L^2(\Gamma_C)), \label{convz}\\
   & \gamma^\eps(\vu^\eps\cdot \bn)\rightharpoonup \eta \quext{weakly in }\,\mathcal H',\label{conv:betadual}\\
   & \beta^\eps(z^\eps)\rightharpoonup \xi_1 \quext{weakly star in }\,L^\infty(0,T;L^2(\Gamma_C)),\label{conv:xi1}\\
   &\alpha^\eps(z_t^\eps)\rightharpoonup \xi_2 \quext{weakly star in }L^\infty(0,T;L^2(\Gamma_C)),\label{conv:xi2}
 \end{align}
 \end{subequations}
 together with
 \begin{align}\label{strongu}
   & \vu_t^\eps\rightarrow \vu_t \quext{strongly in }\,L^2(0,T;\bH),\\
  \label{strongu2}
   & \vu_t^\eps(t)\rightharpoonup \vu_t(t) \quext{weakly in }\,\bH,~~\text{for all }\,t\in[0,T].
 \end{align}
 Moreover for all $t\in(0,T)$ there exists $\eta_{(t)} \in\mathcal H_t'$ such that, for 
 the same subsequence of $\eps\searrow 0$ considered before,
 \begin{align}
   & \gamma^\eps(\vu^\eps\cdot \bn)\llcorner_{(0,t)}\rightharpoonup \eta_{(t)}\quext{weakly in }\,\mathcal H_t'.\label{conv:betadual:t}
 \end{align}
\end{lemm}
\begin{proof}
Convergences \eqref{conv1}, \eqref{conv3}, \eqref{conv:betadual} follow from estimates \eqref{est:uz}. 
Convergences \eqref{conv:xi1} and \eqref{conv:xi2} follow from \eqref{est:xi}.
Moreover, \eqref{est:zt2} implies \eqref{convz}.
Let us now show \eqref{strongu} and \eqref{strongu2}. Thanks to \eqref{conv1} and \eqref{conv3}, 
we can apply the generalized form of Aubin-Lions lemma (\cite[Corollary 4]{Si}, 
\cite[Corollary 7.9]{Roubook}) with the triple of spaces $\bV\subset\subset \bH\subset \bV_0'$. 
This provides \eqref{strongu}. To see \eqref{strongu2}, we first observe that such weak limit 
holds true in the space $\bV_0'$ by \eqref{conv3}. Then the claim follows thanks to the 
fact that $\|\vu_t(t)\|_{\bH}\leq M$ for {\sl every} $t\in[0,T]$ by \eqref{conv1}
and \eqref{C1}.

\medskip

It remains to show \eqref{conv:betadual:t}. Let us set $v^\eps:=(z^\eps)^+$.
Then, thanks to \eqref{convz}, there exists a nonnegative function $v$ such that
\begin{align}\label{convv}
  v^\eps\rightharpoonup v \quext{weakly in }\, L^2(0,T;L^2(\Gamma_C)).
\end{align}
Moreover, by \eqref{conv1}-\eqref{conv3}, the Aubin-Lions lemma, 
and \eqref{cont:trace}, we infer
\begin{equation}\label{convu11}
  \vu^\eps \to \vu \quext{strongly in }\, L^r(0,T;L^r(\Gamma_C;\R^3)),
   ~~\text{for all }\,r\in [1,4).
\end{equation}
Combining \eqref{conv1}, \eqref{convv} and \eqref{convu11}, we obtain
\begin{equation}\label{convu12}
  v\ee \vu^\eps \to v \vu \quext{weakly in }\, L^{4/3}(0,T;L^{4/3}(\Gamma_C;\R^3)).
\end{equation}
Now, for all $t\in(0,T)$, let $\psi\in \calH_t$, let $\bpsi:=\psi\bn$
and let $\bphi$ be the extension of $\bpsi$ provided by Lemma~\ref{lemma0}.
Then let us define a functional $\eta_{(t)}\in \calH_t'$ 
as follows:
\begin{align}\nonumber
   & \mbox{}  \duavw{ \eta_{(t)},\psi }_t
    := \int_0^t (  \vu_t, \bphi_t )
    - ( \vu_t(t) , \bphi(t) )
    + ( \vu_1 , \bphi(0) )
    - \int_0^t \int_\Om \TE \bep(\vu):\bep(\bphi)  \\
  \label{def:eta}  
   & \mbox{}~~~~~ 
   - \int_0^t \int_\Om \TV \bep(\vu_t):\bep(\bphi)
   - \int_0^T \igc v\vu \cdot \bphi
    + \int_0^t\langle g, \bphi \rangle.
\end{align}
Now, thanks to \eqref{conv1} and \eqref{convv}, it is seen that the value of $
\duavw{ \eta_{(t)},\psi }_t$ is exactly the limit of
\begin{align}\nonumber
  & \itt \igc \gamma\ee (\vu\ee \cdot \bn) \psi
   = \int_0^t (  \vu_t\ee, \bphi_t )
    - ( \vu_t\ee(t) , \bphi(t) )
    + ( \vu\ee , \bphi(0) )
    - \int_0^t \int_\Om \TE \bep(\vu^\eps):\bep(\bphi)  \\
 \label{convu13}   
   & \mbox{}~~~~~ 
   - \int_0^t \int_\Om \TV \bep(\vu_t^\eps):\bep(\bphi)
   - \int_0^T \igc v^\eps\vu^\eps \cdot \bphi
    + \int_0^t\langle g, \bphi \rangle,
\end{align}
which is obtained integrating \eqref{u:weak} in time, rearranging terms,
and performing some integration by parts. Hence, in particular
$\eta_{(t)}$ is independent of the extension map $\psi\mapsto\bphi$.
We also observe that,
thanks to \eqref{conv1}-\eqref{conv3} and \eqref{convu12},
the \rhs\ of \eqref{convu13} converges (to the \rhs\ of \eqref{def:eta})
with no need of extracting further subsequences.
The thesis follows.
%
%
\end{proof}
\begin{lemm}[Step 4: refined convergence for $\vu$]\label{lemma4}
 There hold the additional strong convergences:
 \begin{align}\label{uforte1}
   & \vu^\eps \rightarrow \vu \quext{strongly in }\,L^2(0,T;\bV),\\
  \label{uforte2}
   & \vu^\eps(t) \rightarrow \vu(t) \quext{strongly in }\,\bV~~\text{for all }\,t\in[0,T].
 \end{align}
 Moreover, the functions $\eta$ and $\eta_{(t)}$ are identified as follows:
 \begin{align}\label{eq:eta}
   & \eta\in\gamma_w(\vu\cdot \bn),\\
  \label{eq:etat}
   & \eta_{(t)}\in\gamma_{w}((\vu\cdot \bn)\llcorner_{(0,t)}).
 \end{align}
\end{lemm}
\begin{proof}
Let us define, analogously with \eqref{defiG}, the convex functional
\begin{equation}\label{JJee}
  G\ee:L^2(0,T;L^2(\Gamma_C)) \to [0,+\infty), 
   \quad G\ee(v):= \iTT\igc \gammaciapo\ee(v).
\end{equation}
%
%
Then, as observed before, the subdifferential
$\de G\ee$ in $L^2(0,T;L^2(\Gamma_C))$ factually coincides
with the graph $\gamma\ee$.
%
%
%
%
On the other hand, we may interpret
the function $\gamma^\eps$ also as a monotone operator from 
$\mathcal H$ into $\mathcal H'$. Indeed, if $v\in \mathcal H$, then it 
results $\gamma^\eps(v)\in L^2(0,T;L^2(\Gamma_C))\subset \mathcal H'$ 
thanks to the Lipschitz continuity of $\gamma^\eps$. Moreover, 
for all $u,v\in \mathcal H$
we have
\begin{equation}\label{identi12}
  \duavw{ \gamma^\eps(u)-\gamma^\eps(v),u-v }
   =\duas{\gamma^\eps(u)-\gamma^\eps(v),u-v}\geq 0.
\end{equation}
Also, for $u\in \mathcal H$, $\gamma^\eps(u)$ belongs
to the subdifferential $\gamma_w\ee$
of (the restriction to $\calH$ of) $G^\eps$ 
at the point $u$ computed with respect to the duality 
pairing between $\calH'$ and $\calH$. Indeed, we have 
\begin{equation}\label{identi13}
  \duavw{ \gamma^\eps(u),v-u }
   = \iTT \igc \gamma^\eps(u) (v-u) 
   \leq \iTT\igc \big( \gammaciapo^\eps(v)-\gammaciapo^\eps(u) \big)
    = G^\eps(v) - G^\eps(u),
\end{equation}
for all $v\in \mathcal H$. In short terms, we have $\gamma\ee\subset \gamma_w\ee$,
where $\gamma_w^\eps$ acts as a maximal monotone operator 
between $\mathcal H$ and $\mathcal H'$. 

Now, thanks to the monotonicity with respect to $\epsi$ of 
the functionals $G\ee$, owing to \cite[Theorem 3.20]{At} and \cite[Theorem 3.66]{At}, 
the maximal monotone operators $\gamma^\eps_w$ 
converge to  $\gamma_w$ in the graph sense:
\begin{align}\label{graphconv}
  \forall\, [x;y]\in \gamma_w,~~\exists\, [x^\eps;y^\eps]\in \gamma^\eps_w
   \quext{such that }\,[x^\eps;y^\eps]\rightarrow [x;y],
\end{align}
where the convergence is intended with respect to the strong topology of $\mathcal H\times \mathcal H'$.

In order to take the limit of the equation for $\vu$, we first observe that \eqref{u:weak},
after integration in time over $(0,T)$, can be equivalently rewritten in the form
\begin{align}\nonumber
  & \duas{\TE \bep(\vu^\eps) , \bep(\bphi)}
   + \duas{\TV\bep(\vu^\eps_t) , \bep(\bphi)}
   + \duas{\gamma^\eps(\vu^\eps\cdot\bn),\bphi \cdot \bn } \\
 \label{eq1w:eps}  
  & \mbox{}~~~~~
  = \duas{\vu\eet,\bphi_t}
   - ( \vu\eet(T), \bphi(T) )
   + ( \vu_1, \bphi(0) )
   - \iTT\igc v\ee \vu^\eps \cdot \bphi
   + \duavw{ \bg, \bphi},
\end{align}
for any function $\bphi\in \calV$, where $v\ee=(z^\eps)^+$.
Let us also notice that the counterpart of \eqref{eq1w:eps} on subintervals 
$(0,t)$ could be stated analogously.
We now aim to let $\epsi\searrow 0$. To start with, 
we take $\bphi=\vu\ee$ in \eqref{eq1w:eps}. Then, integrating
by parts and rearranging terms, we get
\begin{align}\nonumber
  & \duas{\gamma^\eps(\vu^\eps\cdot\bn),\vu\ee \cdot \bn } =
   - \duas{\TE \bep(\vu^\eps) , \bep(\vu\ee)}
   - \frac12 \io \TV \bep(\vu^\eps(T)) : \bep(\vu^\eps(T))
   + \frac12 \io \TV \bep(\vu_0) : \bep(\vu_0)\\
 \label{semi11}  
  & \mbox{}~~~~~
   + \duas{\vu^\eps_{t},\vu\eet}
   - ( \vu\eet(T), \vu\ee(T) )
   + ( \vu_1, \vu_0 )
   - \iTT\igc v\ee |\vu\ee|^2
   + \duavw{ \bg, \vu\ee}.
\end{align}
%
%
%
%
%
Now, recalling \eqref{convu11}-\eqref{convu12}, we may
take the limit $\epsi\searrow 0$ in relation~\eqref{eq1w:eps}. Actually,
using also \eqref{conv1} and \eqref{conv:betadual}, we easily arrive at
\begin{align}\nonumber
  & \duas{\TE \bep(\vu) , \bep(\bphi)}
   + \duas{\TV\bep(\vu_t) , \bep(\bphi)}
   + \duavw{\eta,\bphi \cdot \bn } \\
 \label{eq1w:0}  
  & \mbox{}~~~~~
  = \duas{\vu_t,\bphi_t}
   - ( \vu_t(T), \bphi(T) )
   + ( \vu_1, \bphi(0) )
   - \iTT\igc v \vu \cdot \bphi
   + \duavw{ \bg, \bphi}.
\end{align}
In particular, in view of the fact that $\vu\in \calV$, we may
take $\bphi=\vu$. Proceeding as for \eqref{semi11}, we then have
\begin{align}\nonumber
  & \duavw{\eta,\vu \cdot \bn } 
   = - \duas{\TE \bep(\vu) , \bep(\vu)}
   - \frac12 \io \TV \bep(\vu(T)) : \bep(\vu(T))
   + \frac12 \io \TV \bep(\vu_0) : \bep(\vu_0)\\
 \label{semi14}  
  & \mbox{}~~~~~
   + \duas{\vu_{t},\vu_t}
   - ( \vu_t(T), \vu(T) )
   + ( \vu_1, \vu_0 )
   - \iTT\igc v |\vu|^2
   + \duavw{ \bg, \vu}.
\end{align}
Now, we use the so-called Minty's semicontinuity trick 
in order to identify the function $\eta$. To this aim, we
take the $\limsup$ as $\epsi\searrow 0$ in \eqref{semi11}
and we compare the outcome to \eqref{semi14}. 
Here, the key point stands in dealing with the integral term.
Actually, 
%
%
%
setting
\begin{equation}\label{semi16}
  \Phi:\RR \times \RR^3 \to \RR, \quad
   \Phi(v,\vu):= (v)^+ | \vu |^2,
\end{equation}
we may observe that $\Phi$ is continuous and nonnegative.
Moreover, it is convex in $\vu$ for any
$v\in\RR$. Hence, Ioffe's semicontinuity theorem (cf., e.g., \cite{Ioffe}) yields
\begin{equation}\label{semi15}
  \iTT\igc v |\vu|^2 
   = \iTT\igc \Phi (v,\vu)
   \le \liminf_{\epsi\searrow 0} \iTT\igc \Phi (v\ee,\vu\ee)
    = \liminf_{\epsi\searrow 0} \iTT\igc v\ee |\vu\ee|^2.
\end{equation}
Indeed, we already know that $(v)^+ = v$ and $(v\ee)^+ = v\ee$. 

With \eqref{semi15} at disposal, taking the $\limsup$ of \eqref{semi11},
using \eqref{conv1}, \eqref{strongu}-\eqref{strongu2}, 
\eqref{convu11} and semicontinuity
of norms with respect to weak convergence, and finally comparing with 
\eqref{semi14}, we get
\begin{equation}\label{semi17}
  \limsup_{\epsi\searrow 0} \,\duavw{\gamma^\eps(\vu^\eps\cdot\bn),\vu\ee \cdot \bn } 
    \le \duavw{\eta,\vu \cdot \bn }.
\end{equation}
Thanks to the fact that $\gamma^\eps(\vu^\eps\cdot\bn)\in \gamma\ee_w(\vu^\eps\cdot\bn)$
and to the graph convergence \eqref{graphconv}, this suffices to 
prove~\eqref{eq:eta}, i.e., to identify $\eta$ (cf.~\cite[Prop.~3.59]{At},
see also \cite[Prop.~2.5]{Br}). 

\smallskip

At this point, we are able to reinforce the strong 
convergence of $\vu\ee$ in \eqref{convu11}.
To this aim, let us first notice that, in view of \eqref{graphconv},
there exists a sequence $[x\ee;y\ee]\in \gamma_w\ee$ such that
$[x\ee;y\ee]\to[\vu\cdot\vn;\eta]$ strongly in $\calH \times \calH'$.
Then, by monotonicity of $\gamma\ee_w$, we have
\begin{equation}\label{semi21}
   \duavw{\gamma^\eps(\vu^\eps\cdot\bn),\vu\ee \cdot \bn }
    \ge \duavw{\gamma^\eps(\vu^\eps\cdot\bn),x\ee}
     + \duavw{y\ee,\vu\ee \cdot \bn }
    - \duavw{y\ee,x\ee}.
\end{equation}
Then, taking the $\liminf$, noting that all terms on the \rhs\ pass
to the limit, and using \eqref{semi17}, we may conclude that
\begin{equation}\label{semi17b}
  \lim_{\epsi\searrow 0} \duas{\gamma^\eps(\vu^\eps\cdot\bn),\vu\ee \cdot \bn } 
    = \duavw{\eta,\vu \cdot \bn }.
\end{equation}
At this point, we rearrange terms in \eqref{semi11} rewriting it in the 
following way:
\begin{align}\nonumber
  & \duas{\TE \bep(\vu^\eps) , \bep(\vu\ee)}
   + \frac12 \io \TV \bep(\vu^\eps(T)) : \bep(\vu^\eps(T))
   = - \duas{\gamma^\eps(\vu^\eps\cdot\bn),\vu\ee \cdot \bn } 
   + \frac12 \io \TV \bep(\vu_0) : \bep(\vu_0)\\
 \label{semi11b}  
  & \mbox{}~~~~~
   + \duas{\vu^\eps_{t},\vu\eet}
   - ( \vu\eet(T), \vu\ee(T) )
   + ( \vu_1, \vu_0 )
   - \iTT\igc v\ee |\vu\ee|^2
   + \duavw{ \bg, \vu\ee}.
\end{align}
We also rearrange terms in \eqref{semi14} in a similar way. Then,
taking the $\limsup$ of \eqref{semi11b}, using \eqref{semi17b},
and comparing with (the rearranged) \eqref{semi14}, we obtain
\begin{equation}\label{semi17c}
  \limsup_{\epsi\searrow 0} \bigg( \duas{\TE \bep(\vu^\eps) , \bep(\vu\ee)}
   + \frac12 \io \TV \bep(\vu^\eps(T)) : \bep(\vu^\eps(T)) \bigg)
  \le \duas{\TE \bep(\vu) , \bep(\vu)}
   + \frac12 \io \TV \bep(\vu(T)) : \bep(\vu(T)).
\end{equation}
Using assumption~(a), we then get~\eqref{uforte1}. Then,
repeating the argument on subintervals $(0,t)$, we also obtain
\eqref{uforte2}. As a further consequence, 
combining \eqref{uforte2} with the uniform boundedness
\eqref{est2}, we get
\begin{equation}\label{uforte3}
  \vu\ee \to \vu
   \quext{strongly in }\,L^4(0,T;\bV),
   ~~\text{whence in }\,L^4(0,T;L^4(\Gamma_C;\RR^3)),
\end{equation}
the latter property following from \eqref{cont:trace}.
Relation \eqref{uforte3} will play a key role when we take
the limit of the equation for $z$.

To conclude the proof of the lemma, we need to show~\eqref{incwt}.
This, however, follows simply by repeating the above argument
on the subintervals $(0,t)$, for $t\le T$.
\end{proof}
\begin{lemm}[Step 5: refined convergence for $z$]\label{lemma5}
 There hold the additional strong convergences
 \begin{align}\label{convz1}
   & z^\eps\rightarrow z \quext{strongly in }\,L^2(0,T;L^2(\Gamma_C)),\\
  \label{convz2}
   & z^\eps(t)\rightarrow z(t)\quext{strongly in }L^2(\Gamma_C),~~\text{ for all }\,t\in[0,T].
 \end{align}
\end{lemm}
\begin{proof}
We prove \eqref{convz1} and \eqref{convz2} by adapting an argument due
to Blanchard, Damlamian and Ghidouche \cite[Lemma~3.3]{BDG} (see also
\cite{CoLuSchSte}). 
To start with, applying $A_\eps$ (cf.~\eqref{yosida}),
equation \eqref{dam:pf} may be equivalently rewritten as 
\begin{align}\label{eq3:eps:inv}
  z^\eps_t = A_\eps\Big(a-\frac{1}{2}|\vu^\eps|^2-\beta^\eps(z^\eps)\Big).
\end{align}
For $\delta\in(0,\frac12)$, we subtract from \eqref{eq3:eps:inv} the analogue expression with $\delta$
in place of $\eps$, then we multiply the result by $z^\eps-z^\delta$. After integration on $(0,t)\times\Gamma_C$,
we get
\begin{align}\no
  & \frac{1}{2}\|z^\eps(t)-z^\delta(t)\|_{L^2(\Gamma_C)}^2
   = \int_0^t \igc \Big(A_\eps\big(a-\frac{1}{2}|\vu^\eps|^2-\beta^\eps(z^\eps)\big)
      - A_\eps\big(a-\frac{1}{2}|\vu^\eps|^2-\beta^\delta(z^\delta)\big)\Big)
          (z^\eps-z^\delta)\\
 \no
  & \mbox{}~~~~~
   + \int_0^t \igc \Big(A_\eps\big(a-\frac{1}{2}|\vu^\eps|^2-\beta^\delta(z^\delta)\big)
      - A_\eps\big(a-\frac{1}{2}|\vu^\delta|^2-\beta^\delta(z^\delta)\big)\Big)(z^\eps-z^\delta )\\
 \no
  & \mbox{}~~~~~
  + \int_0^t\igc \Big(A_\eps\big(a-\frac{1}{2}|\vu^\delta|^2-\beta^\delta(z^\delta)\big)
       -A_\delta\big(a-\frac{1}{2}|\vu^\delta|^2-\beta^\delta(z^\delta)\big)\Big)(z^\eps-z^\delta)\\
 \label{z:strong}
  & =: I_1(t)+I_2(t)+I_3(t).
\end{align}
By the nonexpansivity of $A_\eps$, we first compute
\begin{align}\no
  I_2(t) & \leq \int_0^t \igc |\vu^\eps+\vu^\delta||\vu^\eps-\vu^\delta||z^\eps-z^\delta|\\
 \no
  & \leq \|\vu^\eps-\vu^\delta\|_{L^2(0,t;L^4(\Gamma_C;\RR^3))}\|\vu^\eps+\vu^\delta\|_{L^\infty(0,t;L^4(\Gamma_C;\RR^3))}  
         \|z^\eps-z^\delta\|_{L^2(0,t;L^2(\Gamma_C))}\\
 \label{z:strong_1}
  & \leq c \|\vu^\eps-\vu^\delta\|_{L^2(0,t;L^4(\Gamma_C;\RR^3))}^2
   + \|z^\eps-z^\delta\|_{L^2(0,t;L^2(\Gamma_C))}^2.
\end{align}
Next, recalling \eqref{yosida},
\begin{align}\no
  I_3(t) & = \int_0^t\igc \Big( \frac{\eps}{\eps+1} - \frac{\delta}{\delta+1} \Big) \Big( a -\frac{1}{2}|\vu^\delta|^2-\beta^\delta(z^\delta) \Big)^+(z^\eps-z^\delta)\\
 \no
  & \leq c ( \epsi + \delta ) \big( 1 + \| \vu^\delta\|^2_{L^\infty(0,t;L^4(\Gamma_C;\RR^3))} + \| \beta^\delta(z^\delta) \|_{L^\infty(0,t;L^2(\Gamma_C))} \big)
    \|z^\eps-z^\delta\|_{L^1(0,t;L^2(\Gamma_C))}\\
 \label{nuovo11}
  & \leq c ( \epsi + \delta ) \|z^\eps-z^\delta\|_{L^1(0,t;L^2(\Gamma_C))},
\end{align}
the last inequality following from the previous uniform estimates.

In order to estimate $I_1(t)$ we argue as in \cite[p.~270]{CoLuSchSte}. 
Namely, we use the identities
\begin{align}
  & z^\eps - \eps\beta^\eps(z^\eps) = R_\eps(z^\eps),\label{bre1}\\
  & \beta^\eps(z^\eps)\in\beta(R_\eps(z^\eps))\label{bre2},
\end{align}
valid for all $\eps\in(0,1)$, where $R_\eps$ is the {\sl resolvent}\/ of $\beta$ at
step $\eps$ (see \cite[p.~27-28]{Br}). We have
\begin{align}\no
 I_1(t) & = I_{1,1} + I_{1,2} :=  \int_0^t \int_{\Gamma_{C,0}} \frac{A_\eps(a-\frac{1}{2}|\vu^\eps|^2-\beta^\eps(z^\eps))
    - A_\eps(a-\frac{1}{2}|\vu^\eps|^2-\beta^\delta(z^\delta))}{\beta^\eps(z^\eps)-\beta^\delta(z^\delta)} \times\\
  \no 
  & \mbox{}~~~~~~~~~~~~~~~   
    \times (R_\eps(z^\eps)-R_\delta(z^\delta)) (\beta^\eps(z^\eps)- \beta^\delta(z^\delta)) \\
 \label{z:strong_3}
 & \mbox{}
  + \int_0^t \igc \Big(A_\eps\big(a-\frac{1}{2}|\vu^\eps|^2-\beta^\eps(z^\eps)\big)
   - A_\eps\big(a-\frac{1}{2}|\vu^\eps|^2-\beta^\delta(z^\delta)\big)\Big) (\eps\beta^\eps(z^\eps)-\delta\beta^\delta(z^\delta)),
\end{align}
where $\Gamma_{C,0}$ is the set where the denominator does not vanish.
Now, $I_{1,1}$ is easily seen to be nonpositive thanks to 
the monotonicity of $A_\eps$ and to \eqref{bre1}. The second term $I_{1,2}$
is controlled as follows:
\begin{align}\label{I12}
  | I_{1,2} | 
   & \le \| \beta^\eps(z^\eps)- \beta^\delta(z^\delta) \|_{L^2(0,t;L^2(\Gamma_C))}
   \big ( \| \eps \beta^\eps(z^\eps) - \delta \beta^\delta(z^\delta) \|_{L^2(0,t;L^2(\Gamma_C))} \big) 
    \le c ( \eps + \delta ), 
\end{align}
where we used \eqref{est:beta:z} twice. Gathering the estimates \eqref{z:strong_1}-\eqref{I12}, 
and setting $\phi(t):=\|z^\eps(t)-z^\delta(t)\|_{L^2(\Gamma_C)}$ 
we obtain, for every $t\in (0,T]$ and every $s\in (0,t]$,
the differential inequality
\begin{equation}\label{I13a}
  \frac12 \phi^2(s)
   \leq \frac{c_1^2}2 \Big( \eps + \delta + \|\vu^\eps-\vu^\delta\|_{L^2(0,t;L^4(\Gamma_C;\RR^3))}^2 
         + \| z^\eps-z^\delta\|_{L^2(0,t;L^2(\Gamma_C))}^2 \Big)
     + c_2 (\eps + \delta) \int_0^s \phi(\cdot),
\end{equation}
for suitable $c_1$, $c_2$ independent of $\epsi$ and $\delta$.
Then, applying the generalized Gronwall lemma \cite[Lemme A.5]{Br}, we deduce
\begin{equation}\label{I13b}   
  \phi(s) \leq c_1 \Big( \eps + \delta + \|\vu^\eps-\vu^\delta\|_{L^2(0,t;L^4(\Gamma_C;\RR^3))}^2 
         + \| z^\eps-z^\delta\|_{L^2(0,t;L^2(\Gamma_C))}^2 \Big)^{1/2}
     + c_2 (\eps + \delta) s,
\end{equation}
whence, squaring, and choosing $s=t$,
\begin{equation}\label{I13c}   
  \phi^2(t) \leq c_3 \bigg( \eps + \delta + \|\vu^\eps-\vu^\delta\|_{L^2(0,t;L^4(\Gamma_C;\RR^3))}^2 
         + \int_0^t \phi^2(\cdot) \bigg).
\end{equation}
Applying once more Gronwall's lemma, now in its standard form, 
and recalling \eqref{uforte3}, we then arrive at
\begin{equation}\label{I14}
  \| z^\eps(t)-z^\delta(t) \|_{L^2(\Gamma_C)} \rightarrow 0
   \quext{as }|\eps+\delta| \to 0,
\end{equation}
for every $t\in[0,T]$, which implies \eqref{convz1} and \eqref{convz2}
in view of~\eqref{convz}.
\end{proof}
\begin{lemm}[Step 6: limit flow rule]\label{lemma6}
 The limit functions provided by\/ {\rm Lemma~\ref{lemma3}}
 satisfy condition\, {\rm (ii)} of\/ {\rm Def.~\ref{weaksol}}. 
 Moreover the inclusions \eqref{inc1} and \eqref{inc2} hold.
\end{lemm}
\begin{proof}
Using \eqref{convz}, \eqref{conv:xi1}, \eqref{conv:xi2}, and \eqref{uforte3},
we can take the limit in equation \eqref{damep} and get back \eqref{eq:z},
or, in other words, condition~(ii). Hence, it remains to identify
$\xi_1$ and $\xi_2$. Firstly, we observe that 
inclusion~\eqref{inc1} follows by combining \eqref{convz1} with \eqref{conv:xi1}
and using, e.g., \cite[Prop.~1.1, p.~42]{Ba}. Next, to prove \eqref{inc2},
we use once more Minty's trick (cf., e.g., \cite[Prop.~2.5]{Br}); namely,
we need to check that
\begin{align}\label{4.43}
 \limsup_{\eps\rightarrow0} \int_0^T\igc \alpha_\eps(z_t^\eps) z^\eps_t 
  \leq \int_0^T\igc \xi_2 z_t.
\end{align}
Then, we multiply \eqref{dam:pf} by $z^\eps_t$ to obtain
\begin{equation}\label{4.44}
  \int_0^T \igc \alpha_\eps(z_t^\eps) z^\eps_t 
   = - \|z^\eps_t\|_{L^2(0,T;L^2(\Gamma_C))}^2
    - \igc\betaciapo\ee(z^\eps(T))
    + \igc\betaciapo\ee(z_0)
    + \int_0^T\igc \Big(a-\frac{1}{2}|\vu^\eps|^2\Big)z^\eps_t.
\end{equation}
Then, taking the $\limsup$, we may observe that the four terms on 
the \rhs\ can be managed, respectively, by \eqref{convz} and 
semicontinuity, by \eqref{convz} with
the Mosco-convergence of $\betaciapo\ee$ to
$\betaciapo$ (cf.~\cite[Chap.~3]{At}), by the monotone convergence
theorem, and by combining \eqref{convz} with \eqref{uforte3} (note
that having {\sl strong}\/ convergence of $\vu\ee$
in $L^4$ is essential at this step).
As a consequence, we then infer
\begin{equation}\label{4.45}
  \limsup_{\epsi\searrow 0} \int_0^T \igc \alpha_\eps(z_t^\eps) z^\eps_t 
   \le - \|z_t\|_{L^2(0,T;L^2(\Gamma_C))}^2
    - \igc\betaciapo(z(T))
    + \igc\betaciapo(z_0)
    + \int_0^T\igc \big(a-\frac{1}{2}|\vu|^2\big)z_t,
\end{equation}
Now, testing \eqref{eq:z} by $z_t$ and applying the chain rule formula
\cite[Lemme~3.3, p.~73]{Br} to integrate the product term $\xi_1 z_t$, 
we see that the above \rhs\ is equal
to $\int_0^T\igc \xi_2 z_t$. Hence, \eqref{4.43} follows. In turn, this
implies \eqref{inc2}, which concludes the proof.
\end{proof}
\begin{lemm}[Step 7: condition (i)]\label{lemma7}
 The functions $\vu$, $z$, $\eta$, and $\eta_{(t)}$ obtained in~{\rm Lemma~\ref{lemma3}}
 satisfy\/ {\rm condition (i)} of\/ {\rm Def.~\ref{weaksol}}. 
\end{lemm}
\begin{proof}
The statement follows from all the results obtained so far.
In particular, we notice that the function $v$ in \eqref{convv} 
coincides with $(z)^+$, thanks to the strong convergence \eqref{convz1}. 
Moreover, as a consequence of \eqref{inc1}, we have $z\ge 0$ a.e.~on $(0,T)\times\Gamma_C$.
This implies that $(z)^+=z$. Hence, we may take the limit as $\epsi\searrow 0$
in equation \eqref{eq1w:eps}, which gives back \eqref{eq1w}. Analogously,
repeating the procedure on a subinterval~$(0,t)$, we obtain~\eqref{eq1wt}.
Finally, taking a test function $\bphi\in \calV_t$ with $\bphi(t)\equiv 0$
and considering its trivial
extension $\Ov{\bphi}$ to the whole of $(0,T)$, plugging $\Ov{\bphi}$ in 
\eqref{eq1w:eps} and $\bphi$ in the analogue of \eqref{eq1w:eps} over
the interval~$(0,t)$, and finally letting $\epsi\searrow 0$, we obtain
\eqref{compaeta}.
\end{proof}
\begin{lemm}[Step 8: energy inequality]\label{lemma8} 
 The energy inequality \eqref{en:ineq} holds for almost every $t_1\in [0,T]$
 and for every $t_2 \in (t_1,T]$.
\end{lemm}
\begin{proof}
We know that the approximate solution $(\vu^\eps,z^\eps)$ satisfies the 
energy equality~\eqref{energyee}. Integrating it in the time interval $[t_1,t_2]$, we get
\begin{equation}\label{8.10}
  \calE\ee(t_2) + \int_{t_1}^{t_2} \calD\ee(\cdot)
   = \calE\ee(t_1) 
   + \int_{t_1}^{t_2} \duav{\bg,\bu\eet}
   - \int_{t_1}^{t_2} \igc (z\ee)^- \vu\ee \cdot \vu\eet,
\end{equation}
for every $0\le t_1 \le t_2 \le T$. Now, we take the $\liminf$ on both hand
sides. Then, standard semicontinuity arguments and the convergence relations 
proved so far permit us to prove that
\begin{equation}\label{8.11}
  \calE(t_2) + \int_{t_1}^{t_2} \calD(\cdot)
   \le \liminf_{\eps\searrow 0} 
      \Big( \calE\ee(t_2) + \int_{t_1}^{t_2} \calD\ee(\cdot) \Big)
\end{equation}
for every $t_1$, $t_2$. Moreover, in view of the fact that $z\ge 0$ almost
everywhere, it is easy to check that $(z\ee)^-$ tends to $0$ 
strongly in $L^2(0,T;L^2(\Gamma_C))$, whence the last integral in 
\eqref{8.10} vanishes in the limit. Moreover, it is also easily 
seen that the term depending on $\bg$ passes to the limit. 

Hence, it just remains to control $\calE\ee(t_1)$, which is the more 
delicate point. Indeed, what we want to do is taking the $\liminf$ of
\begin{align}\no
  \calE\ee(t_1)
   & = \io \Big( \frac12 | \vu\ee_t(t_1) |^2 
   + \frac12 \TE \bep(\vu\ee(t_1)) : \bep(\vu\ee(t_1)) \Big)
   + \igc \Big( - a z\ee(t_1) + \frac12 z\ee(t_1) | \vu\ee(t_1) |^2 \Big)\\
 \label{8.12}
  & \mbox{}~~~~~~~~~~  
    + \igc \Big( \betaciapo\ee(z\ee(t_1)) 
      + \gammaciapo\ee ( \vu\ee(t_1) \cdot \vn ) \Big).
\end{align}
Then, in view of \eqref{uforte1}-\eqref{uforte2},
\eqref{strongu}, \eqref{uforte3} and \eqref{convz}, the $\liminf$ 
of the first two integrals is in fact a true limit and is given by
the same couple of integrals rewritten without the~$\epsi$.
On the other hand, this holds just for {\sl almost}\/ every $t_1$; 
indeed, to take the limit of the integral of
$| \vu\ee_t(t_1) |^2 $, we need to use \eqref{strongu};
in other words, we deduce a.e.~(in time) convergence
from $L^2$-convergence by extraction of a subsequence.

Finally, we need to control the last integral in~\eqref{8.12}. 
This can be done by following closely the argument devised in \cite[Sec.~3]{BRSS}, 
to which we refer the reader. Actually, by that method we may prove 
more precisely that
\begin{equation}\label{8.13}
  \limsup_{\epsi\searrow0} \igc \Big( \betaciapo\ee(z\ee(t_1)) 
      + \gammaciapo\ee ( \vu\ee(t_1) \cdot \vn ) \Big)
    \le \igc \Big( \betaciapo(z(t_1)) 
      + \gammaciapo ( \vu(t_1) \cdot \vn ) \Big),
\end{equation}
again, for {\sl almost}\/ every $t_1$. This concludes the proof.
\end{proof}
\beos\label{lemma9} 
 In addition to \eqref{en:ineq}, we may notice that the energy functional 
 defined in \eqref{defiE} is lower semicontinuous with respect to the variable $t\in[0,T]$.
 Indeed, as a consequence of \eqref{reg:u}-\eqref{reg:z}, we have 
 \begin{equation}\label{weakcont}
    \vu\in C([0,T];\bV), \quad
     \vu_t\in C_w([0,T];\bH), \quad
     z\in C([0,T];L^2(\Gamma_C)).
 \end{equation}
 Hence, the lower semicontinuity of 
 $t\mapsto \calE(\vu(t),\vu_t(t),z(t))$ 
 follows easily from the
 convexity of $\betaciapo$ and $\gammaciapo$ and from lower semicontinuity
 of norms with respect to weak convergence.
\eddos

\section*{Acknowledgements}

The financial support of the FP7-IDEAS-ERC-StG \#256872
(EntroPhase) is gratefully acknowledged by the authors. The present paper 
also benefits from the support of the MIUR-PRIN Grant 2010A2TFX2 ``Calculus of Variations''
for GS, and of the GNAMPA (Gruppo Nazionale per l'Analisi Matematica, 
la Probabilit\`a e le loro Applicazioni) of INdAM 
(Istituto Nazionale di Alta Matematica). 
%



\end{document}